\documentclass[11pt]{article}

\usepackage{amsmath, amssymb, amsthm, verbatim,enumerate,bbm,color} 
\numberwithin{equation}{section}

\usepackage{mathrsfs}

\usepackage[backref,colorlinks,citecolor=blue,bookmarks=true]{hyperref}

\title{Testing versus estimation of graph properties, revisited}

\author{
Lior Gishboliner\thanks{ETH Z\"urich. Email: lior.gishboliner@math.ethz.ch. Supported by SNSF grant 200021\_196965.} 
\and 
Nick Kushnir\thanks{School of Mathematics, Tel Aviv University, Tel Aviv 69978, Israel. Email: nickkushnir$@$mail.tau.ac.il.} 
\and 	
Asaf Shapira \thanks{School of Mathematics, Tel Aviv University, Tel Aviv 69978, Israel. Email: asafico$@$tau.ac.il. Supported in part by ERC Consolidator Grant 863438 and NSF-BSF Grant 20196.} 
}

\date{\today}
\allowdisplaybreaks

\theoremstyle{plain}
\newtheorem{theorem}{Theorem}[section]
\newtheorem{lemma}[theorem]{Lemma}
\newtheorem{claim}[theorem]{Claim}
\newtheorem{proposition}[theorem]{Proposition}

\newtheorem{corollary}[theorem]{Corollary}

\newtheorem{problem}[theorem]{Problem}
\newtheorem{remark}[theorem]{Remark}
\newtheorem{definition}[theorem]{Definition}

\makeatletter
\def\moverlay{\mathpalette\mov@rlay}
\def\mov@rlay#1#2{\leavevmode\vtop{%
   \baselineskip\z@skip \lineskiplimit-\maxdimen
   \ialign{\hfil$\m@th#1##$\hfil\cr#2\crcr}}}
\newcommand{\charfusion}[3][\mathord]{
    #1{\ifx#1\mathop\vphantom{#2}\fi
        \mathpalette\mov@rlay{#2\cr#3}
      }
    \ifx#1\mathop\expandafter\displaylimits\fi}
\makeatother

\makeatletter
\renewenvironment{proof}[1][\proofname]
{\par\pushQED{\qed}
	\normalfont\topsep6\p@\@plus6\p@\relax\trivlist
	\item[\hskip\labelsep\bfseries#1\@addpunct{.}]
	\ignorespaces}
{\popQED\endtrivlist\@endpefalse}
\makeatother

\newcommand{\eps}{\varepsilon}
\newcommand{\g}{\gamma}

\RequirePackage{color}\definecolor{RED}{rgb}{1,0,0}\definecolor{BLUE}{rgb}{0,0,1} 

\begin{document}
\date{}
\maketitle

\begin{abstract}

A graph $G$ on $n$ vertices is $\varepsilon$-far from property ${\cal P}$ if one should add/delete at least $\varepsilon n^2$ edges to turn
$G$ into a graph satisfying ${\cal P}$. A {\em distance estimator} for ${\cal P}$ is an algorithm that given $G$ and $\alpha, \varepsilon >0$ distinguishes between the case that $G$ is $(\alpha-\varepsilon)$-close to $\mathcal{P}$ and the case that $G$ is $\alpha$-far from $\mathcal{P}$.
If ${\cal P}$ has a distance estimator whose query complexity depends only on $\varepsilon$, then ${\cal P}$ is said to be {\em estimable}.

Every estimable property is clearly also testable, since testing corresponds to estimating with $\alpha=\varepsilon$.
A central result in the area of property testing is the Fischer--Newman theorem, stating that an
inverse statement also holds, that is, that every testable property is in fact
estimable. The proof of Fischer and Newmann was highly ineffective, since it incurred a tower-type loss when transforming
a testing algorithm for ${\cal P}$ into a distance estimator. This raised the natural problem, studied recently by Fiat--Ron and by Hoppen--Kohayakawa--Lang--Lefmann--Stagni,
whether one can find a transformation with a polynomial loss.
We obtain the following results.
\begin{itemize}
\item We show that if ${\cal P}$ is hereditary, then one can turn a tester for ${\cal P}$ into a distance estimator with an exponential loss.
  This is an exponential improvement over the result of Hoppen et. al., who obtained a transformation with a double exponential loss.
\item We show that for every ${\cal P}$, one can turn a testing algorithm for ${\cal P}$ into a distance estimator with a double exponential loss. This improves over the transformation of Fischer--Newman that incurred a tower-type loss.
\end{itemize}
Our main conceptual contribution in this work is that we manage to turn the approach of Fischer--Newman, which was inherently ineffective, into an efficient one. On the technical level, our main contribution is in establishing certain properties of Frieze--Kannan Weak Regular partitions that are of independent interest.
\end{abstract}



\section{Introduction}\label{sec:intro}

\subsection{Background on graph property testing}

Property testers are fast randomized algorithms that can distinguish between objects satisfying some predetermined property
${\cal P}$ and those that are $\varepsilon$-far from satisfying ${\cal P}$. In most cases, $\varepsilon$-far means that an
$\varepsilon$-proportion of the object's representation needs to be changed in order to obtain a new object satisfying ${\cal P}$.
Hence, testing for ${\cal P}$ is a relaxed version of the classical decision problem which asks to decide whether an object satisfies ${\cal P}$. In this paper we study properties of graphs in the so called {\em adjacency matrix model} (which is also sometimes referred to as the {\em dense graph model}). This is arguably one of the most well studied models in the area of property testing.
The reader is referred to \cite{Gol} for more background and references on property testing.

We now introduce the model of testing graph properties in the adjacency matrix model.
A graph property ${\cal P}$ is a family of graphs closed under isomorphism.
A graph $G$ on $n$ vertices is $\varepsilon$-far from ${\cal P}$ if one should add/delete at least $\varepsilon n^2$ edges to turn
$G$ into a graph satisfying ${\cal P}$. If $G$ is not $\varepsilon$-far from ${\cal P}$ then it is $\varepsilon$-close to ${\cal P}$.
A {\em tester} for ${\cal P}$ is a randomized algorithm that given $\varepsilon >0$ distinguishes with high probability (say, $2/3$)
between graphs satisfying ${\cal P}$ and those that are $\varepsilon$-far from ${\cal P}$. We assume the algorithm can
query for each $1 \leq i,j \leq n$ whether the input $G$ contains the edge $(i,j)$. The {\em edge query complexity}, denoted $Q(\varepsilon)$, of a tester is the
number of edge queries it performs. If ${\cal P}$ has a tester whose edge query complexity depends only on $\varepsilon$ (and is independent of $n$) then ${\cal P}$ is called {\em testable}. In what follows we will mainly work with {\em vertex query complexity} which is the smallest
$q=q(\varepsilon)$ so that we can $\varepsilon$-test ${\cal P}$ by inspecting a subgraph of the input graph $G$, induced by a set of $q$ randomly
selected vertices. By a theorem of Goldreich and Trevisan \cite{GT} we know that $q(\varepsilon) \leq 2 Q(\varepsilon) \leq q^2(\varepsilon)$.
In most (but not all) discussions below we will not care much about these quadratic factors. In such cases we might use the term
{\em query complexity} without mentioning if this is vertex or edge query complexity.

Property testing in the adjacency matrix model was first introduced by Goldreich, Goldwasser and Ron \cite{GGR},
who proved that every {\em partition property} (e.g. $k$-colorability and MAX-CUT) is testable.
There are several general results guaranteeing that a graph property is testable \cite{AFNS09,BCL+06}.
A result of this nature was obtained by Alon and Shapira \cite{AS} who proved that every hereditary\footnote{A graph property is hereditary if it is closed under vertex removal. Some examples are being $3$-colorable, being triangle-free and being induced $H$-free, for some fixed $H$.} graph property is testable. Their proof applied Szemer\'edi's regularity lemma \cite{Sze78} (see also \cite{RSsurvey}), which is one of the most useful tools when studying properties of dense graphs. Using this tool comes with a hefty price, since the bounds one obtains when using the regularity lemma are of
tower-type\footnote{The tower function $\mbox{tower}(x)$ is a tower of exponents of height $x$.}.

One of the central open (meta) problems related to testing graph properties is when can one turn an ineffective (e.g. one with tower-type bounds) result into an efficient one, preferably with polynomial bounds. While this is a quantitative question, what lies beneath it is in fact the following qualitative problem; when can we prove a testability result while avoiding Szemer\'edi's regularity lemma, either by giving a direct combinatorial argument or by using a weaker variant of the regularity lemma (e.g. the Frieze--Kannan regularity lemma \cite{FK-1} which we discuss below). For example,
R\"odl and Duke \cite{RD} used the regularity lemma in order to (implicitly) prove that $k$-colorability is testable. The tower-type bounds
obtained in \cite{RD} were improved to polynomial in \cite{GGR} using a direct argument which avoided the use of the regularity lemma.
A specific central open problem, due to Alon and Fox \cite{AF15a}, concerns hereditary properties, and asks which hereditary
properties are testable with query complexity $\mbox{poly}(1/\varepsilon)$. A systematic investigation of this problem was carried out in~\cite{GS17}.

\subsection{Distance estimation}

In the dense graph model we say that a graph's distance from ${\cal P}$ is $\alpha$, if $\alpha$ is the smallest real so that $G$ is $\alpha$-close to ${\cal P}$.
In other words, this is the minimum number of edges one should add/delete in order to obtain a graph satisfying ${\cal P}$, normalised by $n^2$.
We denote this quantity by $\mathrm{dist}_{\cal P}(G)$.
A {\em distance estimator} for ${\cal P}$ is a randomized algorithm that given $\alpha,\varepsilon >0$ distinguishes with high probability (say, $2/3$)
between graphs that are $(\alpha-\varepsilon)$-close to ${\cal P}$ and those that are $\alpha$-far from ${\cal P}$. If for every $\alpha,\varepsilon$
there is a distance estimator for ${\cal P}$ whose query complexity depends only on $\varepsilon$, then ${\cal P}$ is said to be {\em estimable}.
Note that testing ${\cal P}$ is equivalent to distance estimation with $\alpha=\varepsilon$, hence this notion is at least as strong as testability.

Distance estimation was first studied in \cite{PRR06} and has since been studied in various other settings such as distributions \cite{BFR}, strings \cite{BEK+03},
sparse graphs \cite{CGR13,ELR18,MR09b}, boolean functions \cite{ACCL07,BCE+19}, error correcting codes \cite{GR05,KS09} and image processing \cite{BMR16}.
It is known that in certain settings, there are testable properties which are not estimable \cite{FF06}.
One of the central and most unexpected results in the area of graph property testing is the Fischer--Newman theorem \cite{FN}, which states that
in the setting of graphs, every testable property is also estimable. As with several of the main results in this area, the proof in \cite{FN}
relied on Szemer\'edi's regularity lemma \cite{Sze78} and thus resulted in a tower-type loss when transforming a tester
for ${\cal P}$ into a distance estimator for ${\cal P}$.
Returning to the discussion in the last paragraph of the previous subsection, it is natural to ask if one can improve the transformation
of \cite{FN} and turn a tester for ${\cal P}$ into a distance estimator with a polynomial loss.

\subsection{New results concerning hereditary graph properties}

As we mentioned in the previous subsection, the family of hereditary graph properties has been extensively studied within the setting
of graph property testing. The fact that every hereditary property is testable follows from the following statement, where we use
$ind(F,G)$ to denote the probability that a random mapping $\varphi:V(F)\to V(G)$ is an injective induced homomorphism. \footnote{A mapping  $\varphi:V(F)\to V(G)$ is an induced homomorphism if $uv\in E(F)$ if and only if $\varphi(u)\varphi(v)\in E(G)$.}

\begin{lemma}[Induced Removal Lemma, \cite{AS}]\label{lem:induced}
For every $\varepsilon>0$ and every hereditary ${\cal P}$, there exists $M=M_{\ref{lem:induced}}(\varepsilon,{\cal P})$, $\delta=\delta_{\ref{lem:induced}}(\varepsilon,{\cal P})>0$ and $n_0=n_{{\ref{lem:induced}}}(\varepsilon,{\cal P})$ such that if a graph $G$ on $n\geq n_0$ vertices is $\varepsilon$-far from ${\cal P}$ then there is a graph $F \not \in {\cal P}$ with $|V(F)|\leq M$ such that $ind(F,G )\geq \delta$.
\end{lemma}

The first version of the above lemma was obtained by Alon, Fischer, Krivelevich and Szegedy \cite{AFKS} who proved it when ${\cal P}$ can be characterized using a finite number of forbidden induced subgraphs.
The lemma was proved in full generality by Alon and Shapira \cite{AS}. Alternative proofs were later obtained by Lov\'asz and Szegedy \cite{LS}, Conlon and Fox \cite{CF12} and Borgs et al. \cite{BCL+06}.
It was also extended to the setting of hypergraphs by R\"odl and Schacht \cite{RS}.

Note that it follows immediately from Lemma \ref{lem:induced} that every hereditary property is testable with vertex query complexity
\begin{equation}\label{eqtesting}
q(\eps)=\max\{n_0,M/\delta\}\;.
\end{equation}
Indeed, the algorithm samples a set $X$ of $q$ vertices, queries about all pairs within $X$,
and then accepts if and only if the graph on $X$ satisfies ${\cal P}$.
If $G$ satisfies ${\cal P}$ then the algorithm clearly answers correctly (with probability $1$). If $G$ is $\varepsilon$-far from ${\cal P}$, then by Lemma \ref{lem:induced} a random $M$-tuple of vertices spans an induced copy of a graph $F \not \in {\cal P}$ with probability at least $\delta$. Hence, a sample of size $M/\delta$ contains an induced copy of $F$ with probability at least $2/3$, thus guaranteeing that the sample of vertices does not satisfy ${\cal P}$ (since ${\cal P}$ is hereditary). Recall that \cite{GT} proved that if
${\cal P}$ is testable, then it is testable using an algorithm as above. Hence, the bounds in Lemma \ref{lem:induced}
more or less determine the query complexity of testing a hereditary ${\cal P}$. This raises the following natural problem, introduced by
Hoppen et al. \cite{HKLLS-1,HKLLS-2} and by Fiat and Ron \cite{FR},
asking if it is possible to estimate every hereditary ${\cal P}$ with (roughly) the same query complexity with which it can be tested as in (\ref{eqtesting}).

\begin{problem}\label{q2}
Determine if every hereditary graph property ${\cal P}$ is estimable with query complexity $$n_0 \cdot M/\delta\;,$$
where $M=M_{\ref{lem:induced}}(\varepsilon',{\cal P})$, $\delta=\delta_{\ref{lem:induced}}(\varepsilon',{\cal P})$, $n_0=n_{{\ref{lem:induced}}}(\varepsilon',{\cal P})$
are given by Lemma \ref{lem:induced} with $\varepsilon'=\mathrm{poly}(\varepsilon)$.
\end{problem}

\begin{remark}
There are hereditary graph properties (e.g. triangle-freeness) for which
the best known bounds for $M$ and $\delta$ in Lemma \ref{lem:induced} are of tower-type. One can argue that in such cases
there is little difference between the $\mathrm{tower}(M/\delta)$ bounds given by \cite{FN} and those suggested
by Problem \ref{q2}. However, we should
emphasize that for many of these properties (e.g. triangle-freeness) the tower-type bounds are not known to be tight (indeed, the best known lower bounds are just slightly super polynomial). Perhaps more importantly, there are numerous hereditary graph
properties for which it is known that both $M$ and $\delta$ in Lemma \ref{lem:induced} are polynomial
in $\varepsilon$ (e.g. $k$-colorability, being an interval graph or being a line graph; see the detailed discussion in \cite{GS17}). For all these properties, Problem \ref{q2} suggests a $\mathrm{poly}(1/\varepsilon)$ bound,
versus the $\mathrm{tower}(1/\varepsilon)$ bound given by \cite{FN}.
\end{remark}

Problem \ref{q2} was studied by Hoppen et al. \cite{HKLLS-1,HKLLS-2}.
Their main result was that
every hereditary ${\cal P}$ is estimable with query complexity $2^{\mathrm{poly}((1/\delta)^{M^2}, \log n_0)}$.
Our first main result is the following exponential improvement of this result, making a significant step towards
resolving Problem \ref{q2}.

\begin{theorem}\label{thm:main}
Every hereditary ${\cal P}$ is estimable with query complexity
$$
2^{\mathrm{poly}(M/\delta, \log n_0)}\;,
$$
where $M=M_{\ref{lem:induced}}(\varepsilon/2,{\cal P})$, $\delta=\delta_{\ref{lem:induced}}(\varepsilon/2,{\cal P})$ and $n_0=n_{{\ref{lem:induced}}}(\varepsilon/2,{\cal P})$
are the parameters of Lemma \ref{lem:induced}.
\end{theorem}

\begin{remark}
In all known cases, the best bounds in Lemma \ref{lem:induced} are such
that $\log n_0 \ll 1/\delta$, hence the upper bound of \cite{HKLLS-1} is $2^{(1/\delta)^{O(M^2)}}$ while
the one in Theorem \ref{thm:main} is $2^{\mathrm{poly}(M/\delta)}$.
\end{remark}

In almost all cases, results concerning testing of dense graphs rely on combinatorial statements which
imply trivial algorithms. For example, the algorithm for testing a hereditary property ${\cal P}$
is trivial once we have Lemma \ref{lem:induced} at our disposal. In sharp contrast,
many estimation results involve sampling a set of vertices and then carrying out a highly non-trivial computation over this sample.
This is certainly the case in the present paper, see the proofs of Lemmas \ref{lem:final_part_sig} and \ref{lem:info_from_sig}.
However, thanks to a well known sampling trick \cite{GGR}, one can transfer
any estimation result into a combinatorial statement. For example, this trick
gives the following corollary of Theorem \ref{thm:main}.

\begin{corollary}\label{coro:simple}
Set $q=2^{\mathrm{poly}(M/\delta, \log n_0)}$ as in Theorem \ref{thm:main}. Then
$$
\Pr_X\left[\left|\mathrm{dist}_{\cal P}(G[X])-\mathrm{dist}_{\cal P}(G)\right| \leq \eps\right]\geq 2/3\;,
$$
where the probability is over randomly selected subsets $X$ of $q$ vertices from $G$, and $G[X]$ is the graph induced
by $G$ on $X$.
\end{corollary}

It is interesting to note that with Corollary \ref{coro:simple} at hand, we can now go back and reprove Theorem \ref{thm:main} using
the ``trivial/natural'' algorithm which samples a set of $q$ vertices $X$, computes $\mathrm{dist}_{\cal P}(G[X])$, and then states that
$G$ is $(\alpha-\eps)$-close to ${\cal P}$ if $\mathrm{dist}_{\cal P}(G[X]) \leq \alpha-\eps/2$ and is otherwise $\alpha$-far from ${\cal P}$.

Our proof of Theorem \ref{thm:main} actually gives the bound $2^{\mathrm{poly}(M/\eps\delta, \log n_0)}$. One can speculate
that $\mathrm{poly}(M/\eps\delta)=\mathrm{poly}(M/\delta)$ since in all known cases $\delta$ is at best polynomial
in $\eps$, and in many cases much smaller. In order to formally be able to remove the dependence on $\eps$ from our bound, we
prove the following proposition, where ${\cal P}$ is {\em trivial} if either ${\cal P}$ contains all graphs or if it contains finitely many graphs. The proof of this proposition relies on a subtle application of Ramsey's theorem.

\begin{proposition}\label{easy}
The following holds for every non-trivial hereditary property ${\cal P}$. If $q(\varepsilon)$ denotes the vertex
query complexity of ${\cal P}$ then for every small enough $\varepsilon$, we have
\begin{equation}\label{eqlowerg}
M/\delta \geq q(\varepsilon) \geq  \Omega(1/\eps)\;,
\end{equation}
where $M=M_{\ref{lem:induced}}(\eps,{\cal P})$ and $\delta=\delta_{\ref{lem:induced}}(\eps,{\cal P})$ are the constants of Lemma \ref{lem:induced}.
\end{proposition}
The left inequality above follows from (\ref{eqtesting}).
Observe that the lower bound on $q(\varepsilon)$ is best possible since it is tight when ${\cal P}$ is the
property of having no edges (in which case $q(\varepsilon)=O(1/\varepsilon)$).

It is of course natural to study Problem \ref{q2} also for specific hereditary properties.
A natural problem of this type is whether every hereditary ${\cal P}$ that is testable
with query complexity $\mbox{poly}(1/\varepsilon)$ is also estimable with query complexity
$\mbox{poly}(1/\varepsilon)$. Such an investigation was initiated recently
by Fiat and Ron \cite{FR} who proved such a statement for many natural hereditary properties such as Chordality and not containing an induced path on $4$ vertices.

\subsection{New results concerning general graph properties}

Given the discussion above, the following problem seems natural.

\begin{problem}\label{q1}
Determine if every property ${\cal P}$ that is testable with vertex query complexity $q(\varepsilon)$, is estimable
with query complexity $q(\varepsilon')$ for some $\varepsilon'=\mathrm{poly}(\varepsilon)$.
\end{problem}

Prior to this work, the only result concerning general graph properties ${\cal P}$ was the
transformation of Fischer and Newman \cite{FN} which turns a testing algorithm for a graph
property ${\cal P}$ with query complexity $q(\varepsilon)$ into a distance estimator with
query complexity $\mbox{tower}(q(\varepsilon/2))$. Using the tools we develop in order to obtain
Theorem \ref{thm:main}, we also obtain the following improved bound.

\begin{theorem}\label{thm:main_gen}
If ${\cal P}$ is testable with query complexity $q(\varepsilon)$
then it is estimable with query complexity $2^{\mathrm{poly}(1/\eps)\cdot2^{q(\eps/2)}}$.
\end{theorem}

We would like to argue at this point that since any ``natural'' property satisfies $q(\varepsilon) \geq \log (1/\varepsilon)$ the above
bound can be written as $\mathrm{exp}(\mathrm{exp}(\mathrm{poly}(q(\varepsilon/2))))$.
In order to formally make such a claim, we prove the following variant of Proposition \ref{easy}, in which
${\cal P}$ is {\em unnatural} if there is $\eps_0$ so that the following holds for every $0< \varepsilon < \eps_0$ and $n\geq n_0(\varepsilon)$:
either every $n$-vertex graphs is $\varepsilon$-close to ${\cal P}$, or every $n$-vertex graph does not belong to $\mathcal{P}$.
If ${\cal P}$ is not unnatural then it is (naturally) {\em natural}.

\begin{proposition}\label{easygeneral}
Let ${\cal P}$ be a natural property and let $q(\varepsilon)$ be its vertex query complexity, and $Q(\varepsilon)$ be its edge query complexity. Then
\begin{equation}\label{eqlowerggeneral}
Q(\eps) = \Omega(1/\eps)\;.
\end{equation}
In particular, $q(\varepsilon)=\Omega(\sqrt{1/\varepsilon})$.
\end{proposition}

The ``in particular'' part above follows directly from the Goldreich--Trevisan \cite{GT} theorem mentioned earlier.
Observe that the general lower bound given in (\ref{eqlowerggeneral}) is best possible since it is tight when ${\cal P}$ is the
property of having no edges, where $Q(\varepsilon)=O(1/\varepsilon)$.

\subsection{Main technical contributions and comparison to previous approaches}

\paragraph{Summary of previous approaches:}
The main reason why Szemer\'edi's regularity lemma is so useful when studying testing/estimation problems is
that an $\varepsilon$-regular partition of a graph $G$ determines (approximately) the values of $ind(F,G)$ for all small $F$.
Hence, on a very high level, the way one can estimate a graph's distance to a hereditary property ${\cal P}$ is to take a {\em single} $\varepsilon$-regular partition of $G$ (one such exists by the regularity lemma)
and then try to modify this partition using the smallest possible number of edge modifications, so that the new partition ``predicts'' that there are no induced copies of
graphs $F \not \in {\cal P}$ in the new graph $G'$. A key ``continuity'' feature one has to use at this stage is that if $G$ has a regular partition with certain edge densities between the clusters of the partition,
and one would like to modify $G$ so that in the new graph $G'$ one has a regular partition where the edge densities between the clusters will change on average by $\gamma$, then one can achieve this by modifying $(\gamma+o(1)) n^2$ edges of $G$. Fischer and Newman \cite{FN} critically relied on the fact that regular partitions in the sense of Szemer\'edi have this
continuity property. The approach of \cite{FN} was ineffective since although a regular partition has constant size (i.e., depending only on $\varepsilon$), this constant has tower-type dependence on $\varepsilon$. We should point that one of the key novel ideas of \cite{FN} was a method for obtaining the densities of a single Szemer\'edi partition of the input $G$.

The way Hoppen et al. \cite{HKLLS-1,HKLLS-2} managed to improve upon \cite{FN} (for hereditary ${\cal P}$)
was by first observing that in order to estimate $ind(F,G)$ for all small $F$, one does not need
the full power of Szemer\'edi's regularity lemma. Instead, one can use the weak regularity lemma
of Frieze and Kannan \cite{FK-2} which involves constants that are only exponential in $\varepsilon$.
The main reason why their proof gave a doubly exponential bound is that Frieze--Kannan regular partitions do not
(seem to) have the same continuity feature we mentioned in the previous paragraph with respect to Szemer\'edi partitions.
To overcome this, Hoppen et al. \cite{HKLLS-1,HKLLS-2} introduced a sophisticated method that somehow combines
working with Frieze--Kannan regular partitions in some parts of the proof, together with vertex partitions that have no regularity\footnote{Working with partitions that have no regularity requirements has the advantage that they trivially have the continuity property. Indeed, if we want to change the edge density between two sets $A,B$ by $\gamma$ we just add/remove $\gamma|A||B|$ edges. Needless to say that working with such partitions has various disadvantages resulting from their lack of regularity features.} features at all (these
are sometimes called GGR partitions, after \cite{GGR}) in other parts of the proof.

\paragraph{Our main technical contribution:}
Our main technical contribution in this paper establishes that Frieze-Kannan weak regular partitions ``almost'' satisfy the same
continuity feature we mentioned above with respect to Szemer\'edi partitions. What we show is
that one can indeed efficiently modify a Frieze--Kannan partition if one starts with a partition with guarantees slightly stronger
than those of Frieze--Kannan, and one is content with ending with a usual Frieze--Kannan partition. See Lemma \ref{lem:base_lemma}
for the precise statement, whose proof relies on a randomized-rounding-type argument. With the above continuity feature
at hand, we can now go back to the Fischer--Newman approach and turn it into an effective one, by taking full advantage of the
Frieze--Kannan lemma. One additional hurdle we need
to overcome in order to make sure we only incur an exponential loss in our proof, is a method for finding a Frieze--Kannan partition
of a graph using a constant number of queries. Here we introduce a variant of the method of Fischer--Newman tailored for Frieze--Kannan partitions, see Lemma \ref{lem:final_part_sig}.
The main tools we develop for proving Theorem \ref{thm:main}
turn out to be also applicable for proving Theorem \ref{thm:main_gen}. The reason why in Theorem \ref{thm:main_gen} we
have a double exponential loss is that it is not enough to estimate $ind(F,G)$ for a single $F$ (as in Theorem \ref{thm:main} thanks
to Lemma \ref{lem:induced}) but we instead need to control $ind(F,G)$ for all graphs $F$ of order $q(\varepsilon)$.
We expect Lemmas \ref{lem:final_part_sig} and \ref{lem:base_lemma} to be applicable in future studies related to efficient testing and estimation
of graph properties.

\paragraph{Paper overview:}
In Section \ref{sec:rob_final_and_proof} we introduce the two main lemmas in the paper, and show how they imply Theorem \ref{thm:main}.
These lemmas are proved in Sections \ref{sec:proof_lemma_1} and \ref{sec:proof_lemma_2}.
In Section \ref{sec:conc} we prove Theorem \ref{thm:main_gen}. We prove Proposition \ref{easy} at the end of Section \ref{sec:rob_final_and_proof} and Proposition \ref{easygeneral} at the end of Section \ref{sec:conc}.
We use $a=\mathrm{poly}(x)$ to denote the fact that $a$ is bounded from above (or below, when $0<x<1$) by $x^d$ for some fixed $d$, which is independent of $n$ or $\varepsilon$. Also, when we say that ``for every $a=\mathrm{poly}(x)$
there is $b=\mathrm{poly}(x)$'' we mean that for every $d$ there is $d'$ so that if $a \leq x^d$ then there is a $b \leq x^{d'}$.

\section{The Key Lemmas and Proof of Theorem~\ref{thm:main}}\label{sec:rob_final_and_proof}

Our goal in this section is to state Lemmas \ref{lem:final_part_sig} and \ref{lem:info_from_sig}
and then use them to derive Theorem \ref{thm:main}.
We prove these lemmas in Sections \ref{sec:proof_lemma_1} and \ref{sec:proof_lemma_2}.
At the end of this section we also prove Proposition \ref{easy}.

To state Lemmas \ref{lem:final_part_sig} and \ref{lem:info_from_sig} we need some definitions.
We first recall that given a graph $G=(V,E)$, an equipartition $A = \{V_{1},\dots, V_{k}\}$
of $V(G)$ is a partition satisfying $||V_i|-|V_j||\leq 1$.
Given a graph $G$ and subsets $X,Y\subseteq V(G)$, we use $e(X,Y)$ to denote the number of edges between $X$ and $Y$,
and $d(X,Y)=e(X,Y)/|X||Y|$ to denote the {\em density} between them.

\begin{definition}[Signature]\label{def:sig}
For an equipartition $A = \{V_{1},\dots, V_{t}\}$ of $V(G)$, a $(\gamma, \eps)$-signature of $A$ is a sequence of reals $S = (\eta_{i,j} )_{1\le i<j\le t}$, such that $|d(V_{i}, V_{j} ) - \eta_{i,j} | \le \gamma$ for all but at most $\eps \binom{t}{2}$ of the pairs $i < j$.
A $(\g,\g)$-signature is referred to as $\g$-signature.
\end{definition}

\begin{definition}[Index of a partition]\label{def:index_robust}
For an equipartition $A$ of a graph $V(G)$ into $t$ sets, we define the {\em index} of $A$ to be
$$
ind(A)=\frac{1}{t^{2}}\sum_{1\le i<j\le t}d^{2}(V_{i},V_{j})\;.
$$
\end{definition}

\begin{definition}[Final partition]
For a function $f:\mathbb{N}\to \mathbb{N}$ and $\g>0$, we say that an equipartition $A$ of $G$ consisting of $t$ sets is
$(f, \g)$-final if there exists no equipartition $B$ of $V(G)$ with at least $t$ and up to $f(t)$ sets for which $ind(B) \ge ind(A)+\g\;$.
\end{definition}

The above notion of a final partition is useful since (as we show later) every graph has such a partition and furthermore, we can design
an algorithm for finding a signature of one such partition of an input $G$.
The first key lemma leading to the proof of Theorem \ref{thm:main} does exactly that.

\begin{lemma}\label{lem:final_part_sig}
For every $k,\zeta>0$, and every $\g=\mathrm{poly}(\zeta)$ and $f_{\zeta}(x)=x \cdot 2^{\mathrm{poly}(1/\zeta)}$, there are $q=q_{\ref{lem:final_part_sig}}(\zeta,k)$, $N=N_{\ref{lem:final_part_sig}}(\zeta,k)$ and $T=T_{\ref{lem:final_part_sig}}(\zeta,k)$
so that
$$
q,N,T \leq \mathrm{poly}(k) \cdot 2^{\mathrm{poly}(1/\zeta)}
$$
and such that the following holds. If $G$ is a graph on at least $N$ vertices then there is an algorithm making at most $q$ queries to $G$, computing with probability at least $\frac{2}{3}$ a $\g$-signature of an $(f_{\zeta},\gamma)$-final partition of $G$ into at least $k$ and at most $T$ sets.
\end{lemma}

We prove the above lemma is Section \ref{sec:proof_lemma_1}.
The following is the second key lemma, which we prove in Section \ref{sec:proof_lemma_2}.
In its statement we use the notion $ind(F, G)$ which we defined before the statement of Lemma \ref{lem:induced}.
What it roughly states, is that having a signature of $G$ (with good parameters) is enough for estimating
$G$'s distance to satisfying ${\cal P}$.

\begin{lemma}\label{lem:info_from_sig}
For every $h,\eps,\delta>0$, there are $\g=\g_{\ref{lem:info_from_sig}}(h,\eps,\delta)$,
$s=s_{\ref{lem:info_from_sig}}(h,\eps,\delta)$ and
$f_{\ref{lem:info_from_sig}}^{(h,\eps,\delta)}:\mathbb{N}\to\mathbb{N}$ so that
$$
\g=\mathrm{poly}(\eps\delta/h),~~
s=\mathrm{poly}(h/\eps\delta),~~
f_{\ref{lem:info_from_sig}}(x)= x \cdot 2^{{\mathrm{poly}(h/\eps\delta)}}
$$
and the following holds. For every family $\mathcal{H}$ of graphs, each on at most $h$ vertices, there exists a deterministic algorithm, that receives as an input a $\g$-signature $S$ of an $(f_{\ref{lem:info_from_sig}},\g)$-final partition $A$ into $t\ge s$ sets of a graph $G$ with $n\ge N_{\ref{lem:info_from_sig}}(h,\eps,\delta,t)= \mathrm{poly}(t) \cdot 2^{\mathrm{poly}(h/\eps\delta)}$ vertices, and distinguishes given any $\alpha$ between the following two cases:
\begin{itemize}
\item[(i)] $G$ is $(\alpha-\eps)$ close to some graph $G^{\prime}$ for which $ind(H,G^\prime)=0$ for every $H\in \mathcal{H}$.
\item[(ii)] $G$ is $\alpha$-far from every $G^{\prime}$ for which $ind(H,G^\prime)<\delta$ for every $H\in \mathcal{H}$.
\end{itemize}
\end{lemma}

\begin{proof}[Proof (of Theorem \ref{thm:main})]\label{pf:main_lemma}
Suppose $\mathcal{P}$ is a hereditary graph property, and let $\alpha, \eps > 0 $.
Lemma \ref{lem:induced} with inputs $\eps /2$ and $\mathcal{P}$ asserts that there are
$$
h=M_{\ref{lem:induced}}(\eps/2),~~\delta=\delta_{\ref{lem:induced}}(\eps/2),~~n_0=n_{\ref{lem:induced}}(\eps/2)\;,
$$
so that if a graph $G$ on at least $n_{0}$ vertices is $\eps/2$-far from $\mathcal{P}$, then $ind(H,G)\ge  \delta$ for some $H \notin {\cal P}$ with $|V(H)|\leq h$.
We need to describe an algorithm making $2^{\mathrm{poly}(h/\delta,\log n_0)}$ queries to $G$ and distinguishes with probability at least
$2/3$ between the case that $G$ is $(\alpha-\eps)$-close to ${\cal P}$ and the case that $G$ is $\alpha$-far from ${\cal P}$.
Set
$$
\g=\g_{\ref{lem:info_from_sig}}(h,\eps/2,\delta),~~s=s_{\ref{lem:info_from_sig}}(h,\eps/2,\delta),~~f=f_{\ref{lem:info_from_sig}}^{(h,\eps/2,\delta)}\;.
$$
Finally, set $\zeta=\delta\varepsilon/h$ and observe that
$$
\g=\g_{\ref{lem:info_from_sig}}(h,\eps/2,\delta)=\mathrm{poly}(\eps\delta/2h)=\mathrm{poly}(\zeta)\;,
$$
that
$$
f(x)=f_{\ref{lem:info_from_sig}}^{(h,\eps/2,\delta)}(x)=x\cdot 2^{{\mathrm{poly}(2h/\eps\delta)}} = x\cdot 2^{{\mathrm{poly}(1/\zeta)}}\;,
$$
that
$$
s=s_{\ref{lem:info_from_sig}}(h,\eps/2,\delta) = \mathrm{poly}(2h/\eps\delta)=\mathrm{poly}(1/\zeta)\;.
$$
Also, note that by Proposition \ref{easy} we have $\mathrm{poly}(1/\zeta)=\mathrm{poly}(h/\delta)$.
Let $q,N,T$ be the parameters given by Lemma \ref{lem:final_part_sig} when applied with $k=s$, and $\zeta,\gamma,f$ defined above. (note that $\gamma$ and $f$ satisfy the assumptions of the lemma).
Lemma \ref{lem:final_part_sig} then guarantees that $q,N,T \leq  2^{\mathrm{poly}(1/\zeta)}\leq 2^{\mathrm{poly}(h/\delta)}$.

If $G$ has less than $N$ vertices then we can just ask about all the edges of $G$ and answer correctly with probability $1$.
The number of queries is then at most $N^2\leq 2^{\mathrm{poly}(h/\delta)}$ as needed.
If $G$ has more than $N$ vertices then we can use the algorithm of Lemma \ref{lem:final_part_sig} with the parameters $k,\zeta,\gamma,f$ defined above. The algorithm makes at most $q  \leq 2^{\mathrm{poly}(h/\delta)}$ queries and with probability at least $2/3$ returns a
$\g$-signature $S$ of an equipartition of $G$ into $s \leq t\leq T$ sets that is $(f,\g)$-final.
Let
$$
N'=N_{\ref{lem:info_from_sig}}(h,\eps/2,\delta,T)=\mathrm{poly}(T)\cdot 2^{\mathrm{poly}(h/\eps\delta)}=  2^{\mathrm{poly}(h/\delta)}\;.
$$
Again, if $G$ has less than $N_1=\max\{N',n_0\}$ vertices then we can just ask about all the edges of $G$ and answer correctly with probability $1$.
The number of queries is then at most $(N_1)^2 \leq 2^{\mathrm{poly}(h/\delta,\log n_0)}$ as needed.

Suppose then that $G$ has at least $\max\{N,N_1\}$ vertices.
Let ${\cal H}$ be the family of graph on at most $h$ vertices which do not satisfy ${\cal P}$.
Then we can now run the algorithm of Lemma \ref{lem:info_from_sig} on the signature $S$, with respect to ${\cal H}$,
with $\alpha'=\alpha-\eps/2$ and with $\eps/2$ instead of $\eps$ (note that we chose the parameters with $\eps/2$).
If the algorithm says that case $(i)$ holds (namely that $G$ is $(\alpha'-\eps/2)$-close to some $G'$
with $ind(H,G')=0$ for every $H \in {\cal H}$)
then we declare that $G$ is $(\alpha-\eps)$-close to ${\cal P}$, and if
the algorithm says that case $(ii)$ holds (namely that $G$ is $\alpha'$-far from every $G'$
with $ind(H,G') < \delta$ for every $H \in {\cal H}$) then we declare that $G$ is $\alpha$-far from ${\cal P}$.

Let us prove
the correctness of the above algorithm. If $G$ is $(\alpha-\eps)$-close to ${\cal P}$ then it is $(\alpha-\eps)$-close
to a graph $G'$ satisfying $ind(H,G')=0$ for every $H \in {\cal H}$. Since $\alpha-\eps=\alpha'-\eps/2$
the algorithm will say that
case $(i)$ holds, hence the algorithm answers correctly in this case.
Suppose now that $G$ is $\alpha$-far from ${\cal P}$. Then any $G'$ that is $\alpha'$-close to $G$ must be $\eps/2$-far from
${\cal P}$. Hence, by Lemma \ref{lem:induced} in any such $G'$ we have $ind(H,G') \geq \delta$ for at least one $H \in {\cal H}$.
We conclude that $G$ is $\alpha'$-far from every $G'$
satisfying $ind(H,G') < \delta$ for every $H \in {\cal H}$.
Hence, the algorithm of Lemma \ref{lem:info_from_sig} will say that case $(ii)$ holds , so our algorithm will answer correctly in this case as well.
\end{proof}

\begin{proof}[Proof (of Proposition \ref{easy}):]

Recall that a blowup of a graph $H$ on $h$ vertices is the graph obtained from $H$ by replacing every vertex $i \in V(H)$ with
an independent set of vertices $S_i$, and replacing every edge $(i,j)$ with a complete bipartite
graph between $S_i$ and $S_j$. A $b$-blowup is a blowup where every $S_i$ is of size $b$.
Suppose $H_b$ is a $b$-blowup of $H$ and $f:[h] \rightarrow \{0,1\}$ is a $0/1$ assignment
to $H$'s vertices. Then $H^f_b$ is the graph obtained by taking $H_b$ and then turning every set of vertices
$S_i$ into a clique if and only if $f(i)=1$.

Let us say that $H$ is {\em good} if there is a $b=b(H)$ so that for every $f$ as above, we have $H^f_b \not \in {\cal P}$.
We first observe that if $H \not \in {\cal P}$ then $H$ is good, since we can take $b=1$. We also note
that a single vertex cannot be good, since if a vertex is good, then there must be a clique and an independent set
which do not satisfy ${\cal P}$, implying by Ramsey's theorem, that every large enough graph is not in ${\cal P}$,
contradicting our assumption that ${\cal P}$ is non-trivial.

Let $H$ be a graph not satisfying ${\cal P}$ (one exists since ${\cal P}$ is non-trivial).
By the previous paragraph, $H$ is good. If one of the induced subgraphs of $H$ on $|V(H)|-1$ vertices is also good,
then replace $H$ with this induced subgraph. Suppose $H$ is the (minimally) good graph we end up with.
By the previous paragraph we know that $h \geq 2$.
Let $H'$ be the graph obtained by removing vertex $h$ from $H$. Then $H'$ is not good.

Fix $0< \eps < \frac{1}{10|H|}$ and large $n$ and let $G_n$ be the blowup of $H$ where the vertex set that replaces vertex $h \in V(H)$, call it $V_h$, is of size $\eps n$ and all the other $h-1$ sets, call them $V_1,\ldots,V_{h-1}$, are of equal
size $(n-\eps n)/(h-1)$. Since $H'$ is not good, we know that given $b=(n-\eps n)/(h-1)$ there is an $f':[h-1]\rightarrow \{0,1\}$
so that $H^f_b \in {\cal P}$. For every $1 \leq i \leq h-1$ we turn $V_i$ into a clique if and only if $f(i)=1$.
Observe that every induced subgraph of $G_n$ that has no vertex in $V_h$ satisfies ${\cal P}$.

Let $b_0=b(H)$ be the constant from the definition of a good graph, and let $r(b_0)$ be the Ramsey number of $b_0$, that is,
the smallest integer so that every graph on $r$ vertices has a clique or an independent set on $b_0$ vertices.
We now claim that $G_n$ is $\eps/C$-far from ${\cal P}$, where $C=2h^4r^2$. In fact, we claim that if one changes
less than $\eps n^2/C$ edges between the sets $V_1,\ldots,V_h$ then (no matter what changes one performs within the sets $V_1,\ldots,V_h$) the resulting graph does not satisfy ${\cal P}$. Since ${\cal P}$ is hereditary, it is enough to show that there is still an induced subgraph not satisfying ${\cal P}$. Indeed, consider an $(h\cdot r)$-tuple $v_1,\ldots,v_{h\cdot r}$ of vertices,
obtained by picking, for every $1 \leq i \leq h$, a set of $r$ vertices from $V_i$ uniformly at random.
Fix $i <j$.
Since $|V_i||V_j| \geq \eps n^2/h^2$ then we modified at most a $\frac{1}{2r^2h^2}$ fraction of the pairs between $V_i,V_j$.
Therefore, the probability that our sample contains a modified pair of vertices is at most $\frac{1}{2h^2}$. Hence, by the union bound,
the probability that our sample contains some modified pair of vertices between some pair $V_i,V_j$ is at most $1/2$. We infer that there
is a choice of $h\cdot r$ vertices so that the induced graph on them is an $r$-blowup of $H$. By the choice
of $r$, we can find in this set a $b_0$-blowup of $H$ so that each set of $b_0$ vertices is a clique or an independent set.
Since $H$ is good, this means that this is graph does not satisfy ${\cal P}$.

Now, as we noted in the introduction (see equation (\ref{eqtesting})), by the definition of $M$ and $\delta$, a sample of $M/\delta$ vertices contains, with probability at least
$1/2$, a graph not in ${\cal P}$. As we noted earlier, every subgraph of $G_n$ not containing a vertex from $V_h$
satisfies ${\cal P}$. But to hit $V_h$ with probability at least $1/2$ one must sample at least $1/\varepsilon$ vertices. This means
that we must have $M/\delta \geq 1/\varepsilon$.
\end{proof}

\section{Proof of Lemma \ref{lem:final_part_sig}}\label{sec:proof_lemma_1}

The proof is similar to one in \cite{FN}. What they have shown is that for every $f,\gamma$, one can find an $(f,\gamma)$-final partition
with a constant, albeit huge tower-type, query complexity. What we do here is show that for restricted types of $f$, one can get
a much better bound. To do this we also need to rely on a recent result of \cite{SS}.

\subsection{Preliminary lemmas}

In this subsection we describe some preliminary lemmas that will be used in the next subsection in which we prove Lemma \ref{lem:final_part_sig}.
We will need the following Chernoff-type large deviation inequality.

\begin{lemma}\label{lem:chern}
Suppose $X_{1},\dots,X_{m}$ are $m$ independent Boolean random variables, so that for every $1 \leq i \leq m$ we have $\Pr[X_{i}=1]=p_{i}$. Let $E=\sum_{i=1}^{m}p_{i}$. Then,
$\Pr[|\sum_{i=1}^{m}X_{i}-E|\ge\theta m]\le2e^{-2\theta^{2}m}$.
\end{lemma}


\begin{definition}(Partition Properties)\label{def:ggr_part_prop}
A partition property is a triple $\pi=(s,\ell,u)$ where $s$ is an integer (the size of the partition property),
$\ell$ is a vector of $\binom{s}{2}$ reals $0 \leq \alpha_{i,j} \leq 1$ for each $1 \leq i<j \leq s$, and $u$ is a vector of $\binom{s}{2}$
reals $0 \leq \beta_{i,j} \leq 1$ for each $1 \leq i<j \leq s$.
We say that a graph $G$ satisfies $\pi$ if there is an equipartition $\{V_1,\dots,V_s\}$ of $V(G)$, such that $\alpha_{ij}\le d(V_{i},V_{j})\le\beta_{ij}$ for every $1 \le i<j \le s$.

Given $s$ and $\mu$ we use $\pi(s,\mu)$ to denote the family of partition properties $\pi$ of size $s$ in which
every $\alpha_{i,j}$ and $\beta_{i,j}$ is an integer multiple of $\mu$ (so $\pi(s,\mu)$ contains $\{0,\mu,2\mu,\dots,1\}^{2{s \choose 2}}$ partition properties).
Finally, define $\Pi(t,\mu)=\bigcup_{s \le t}\pi(s,\mu)$
\end{definition}



Note that each $\pi$ as above is one of the partition properties studied in \cite{GGR}, where it was shown that they are $\mu$-testable with query complexity $(1/\mu)^{\text{poly}(s)}$.
This was improved recently to $\mathrm{poly}(s/\mu)$ in \cite{SS}.
The next lemma states that with (roughly) the same query complexity we can in fact {\em simultaneously} test all properties in $\Pi(t,\mu)$.

\begin{lemma}\label{lem:GGR-test}
For every $t$ and $\mu>0$ there is $q=q_{\ref{lem:GGR-test}}(t,\mu)=\mathrm{poly}(t/\mu)$ satisfying the following.
There is a randomized algorithm, that given a graph $G$, makes $q$ queries to $G$ and with probability at least $2/3$, for every $\pi \in \Pi(t,\mu)$, distinguishes between
the case that $G$ satisfies $\pi$ and the case that $G$ is $\mu$-far from $\pi$.
\end{lemma}\label{lem:applied_ggr_test}

\begin{proof}
A result of \cite{SS} states that every $\pi \in \Pi(s,\mu)$ is $\mu$-testable with query complexity $q'=\mathrm{poly}(s/\mu)$.
Set $b=|\Pi(t,\mu)| \leq (1/\mu)^{t^2}$. Fix $\pi \in \Pi(t,\mu)$. If we execute the $\mu$-testing algorithm $10\log(b)$ times and then take the majority outcome\footnote{That is, we are doing the standard error reduction trick.},
then by a standard application of Lemma \ref{lem:chern} we get a new algorithm making $\log b \cdot \mathrm{poly}(s /\mu) \leq \log b \cdot \mathrm{poly}(t /\mu)$ queries, that distinguishes
between the case that $G$ satisfies $\pi$ and the case that $G$ is $\mu$-far from $\pi$, and errs with probability at most $1/3b$. Note that to do this
we may sample a set $Q$ of size $q'\cdot 10\log(b)$ and then execute the standard $\mu$-tester for $\pi$ on $10\log(b)$ disjoint sets of size $q'$ (thus
guaranteeing full independence between the $10\log(b)$ iterations). Since for each $\pi$, the random set $Q$ is such that the algorithm errs with probability
at most $1/3b$, we get by the union bound that the probability that it errs for some $\pi \in \Pi(t,\mu)$ is at most $1/3$. Finally, the query complexity
of this algorithm is $\log b \cdot \mathrm{poly}(t /\mu)=\mathrm{poly}(t /\mu)$.
\end{proof}



\begin{proof}[Proof (of Lemma \ref{lem:final_part_sig}):]
Given $k,\zeta,\g$ and $f_{\zeta}$ as in the statement of the lemma, we define
$T_0=k$ and for $i\ge1$ define $T_i=f_{\zeta}(T_{i-1})$.
Now set the following parameters.
$$
N=N_{\ref{lem:final_part_sig}}(k,\zeta)=T_{2/\g}=k\cdot 2^{\text{poly}(1/\zeta)},~~T=T_{\ref{lem:final_part_sig}}(k,\zeta)=T_{2/\g}=k\cdot 2^{\text{poly}(1/\zeta)}\;,
$$
and
$$
t=f_{\zeta}(T)=k\cdot 2^{\text{poly}(1/\zeta)},~~\mu=\frac{\g}{48(f_{\zeta}(T))^2}=\frac{1}{\mathrm{poly}(k)\cdot 2^{\mathrm{poly}(1/\zeta)}}\;.
$$

We now describe the algorithm for finding a signature $S$ satisfying the requirement of the lemma.
For what follows let $\pi'(s,\mu)$ be the partition properties in which $\beta_{i,j}=\alpha_{i,j}+\mu$ for every $1 \leq i < j \leq s$.
Also for each $\pi \in \pi'(s,\mu)$ define the index of $\pi$ to be $ind(\pi)=\frac{1}{t^{2}}\sum_{1\le i<j\le t}\alpha_{ij}^{2}$.
In the Step-$1$ we run the algorithm of Lemma \ref{lem:GGR-test} with the parameters $t,\mu$ defined above.
This is the only randomized part of the algorithm.
In the Step-$2$ of the algorithm we do the following.
\begin{itemize}
	\item[$(i)$] For each $k \leq s \leq t$ set $M(s)=\max_{\pi}ind(\pi)$ where the maximum is taken over all $\pi \in \pi'(s,\mu)$ which the algorithm of Step-$1$ accepted.
	\item[$(ii)$] Let $s^{\star}$ be the smallest number in $\{k,\dots ,T\}$ such that $M(s^\prime) \leq M(s^\star)+\frac{3}{4}\g$ for every $s^\prime \in \{s^\star +1,\dots ,f_{\zeta}(s^\star)\}$. If there exists such an $s^\star$, output the signature $S^\star$ that achieves the maximum over $s^\star$. Otherwise, the algorithm fails.
\end{itemize}

Note that the query complexity of the algorithm is $q=q_{\ref{lem:GGR-test}}(t,\mu)=\mathrm{poly}(t/\mu)=\mathrm{poly}(k)\cdot 2^{\mathrm{poly}(1/\zeta)}$, as needed.
Also, Lemma \ref{lem:GGR-test} guarantees that Step-$1$ of the above described algorithm succeeds with probability at least $2/3$.
It thus remains to show that assuming this event holds, Step-$2$ of the algorithm will return an $(f_{\zeta},\gamma)$-final partition. First of all note that if it succeeds then it returns a partition of size at least $k$ and at most $T$, as required.

The proof that if Step-$1$ succeeded, then Step-$2$ returns an $(f_{\zeta},\gamma)$-final partition is identical to the proof of Claim 5.5 in \cite{FN}, so we give a sketch of the proof.
First, the reader might be wondering why every graph necessarily has an $(f_{\zeta},\gamma)$-final partition as in the statement of the lemma.
Let us actually explain why every $G$ has an $(f_{\zeta},\gamma/2)$-final partition, while using the definitions we introduced above.
Start from an arbitrary equipartition $A_0$ of $G$ into $T_0=k$ sets, and let $ind_0=ind(A_0)$ denote the index of $A_0$ as in Definition \ref{def:index_robust}. If $A_0$ is $(f_{\zeta},\gamma/2)$-final then we are done.
If not, then there must be another partition $A_1$ of $G$ with at least $T_0$ and at most $f(T_0)=T_1$ parts, with index $ind(A_1) \geq ind(A_0)+\gamma/2$. Since $0 \leq ind(A) \leq 1$ for every equipartition,
we see that this process will eventually end up with a partition $A$ of size $k \leq s \leq T$ so that all partitions of $G$ into at least $s$ and at most $f(s)$ parts have index less than $ind(A)+\gamma/2$.
But this means that $A$ is $(f_{\zeta},\gamma/2)$-final. Note that we thus get that $G$ has a $(f_{\zeta},\gamma/2)$-final partition $A$ of size $s \leq T$.

Let us now explain how to turn the above existential proof into a proof of correctness of the algorithm describe earlier.
Let $M_G(s)$ denote the largest index of an equipartition of $G$ of size $s$.
First we claim that for every $k \leq s \leq t$,
\begin{equation}\label{eq:M_G(s) M(s)}
M(s) - \gamma/8 \leq M_G(s) \leq M(s) + \gamma/8.
\end{equation}
For the second inequality in \eqref{eq:M_G(s) M(s)}, let $A$ be an equipartition with $s$ parts such that $M_G(s) = ind(A)$. Let $\pi \in \pi'(s,\mu)$ be the partition property obtained from $A$ by rounding down the densities to the closest integer multiple of $\mu$. Then we have $|ind(A)-ind(\pi)|\leq 3\mu \leq \gamma/8$. Hence, $M(s) \geq ind(\pi) \geq ind(A) - \gamma/8 = M_G(s) - \gamma/8$.

For the first inequality in \eqref{eq:M_G(s) M(s)}, let $\pi \in \pi'(s,\mu)$ be a partition property which the algorithm accepted and such that $M(s) = ind(\pi)$. Then $G$ must be $\mu$-close to $\pi$ (as otherwise $\pi$ should have been rejected). Let $G'$ be a graph $\mu$-close to $G$ that satisfies $\pi$, and let $A$ be the vertex partition of $G'$ witnessing that $G'$ satisfies $\pi$. Note that when turning $G$ into $G'$, for each pair of parts of $A$, we change the density between this pair by at most $\mu s^2$. Hence, in $G$, the partition property $\pi$ is a $2\mu s^2$-signature of $A$ (here and in what follows, we view $\pi$ as a signature). So $|ind(A) - ind(\pi)| \leq 6\mu s^2 \leq \gamma/8$, using our choice of $\mu$. Now, $M_G(s) \geq ind(A) \geq ind(\pi) - \gamma/8 = M(s) - \gamma/8$. This proves \eqref{eq:M_G(s) M(s)}.


It follows from the existential proof above that there is $k \leq s^{\star} \leq T$ and an equipartition $A$ of $G$ into $s^{\star}$ parts which is $(f_{\zeta},\gamma/2)$-final. We can assume that $M_G(s^{\star}) = ind(A)$, because the equipartition satisfying this must also be final. We have $M_G(s') \leq M_G(s^{\star}) + \gamma/2$ for every $s^{\star} \leq s' \leq f_{\zeta}(s^{\star})$. By \eqref{eq:M_G(s) M(s)}, this implies that $M(s') \leq M(s^{\star}) + 3\gamma/4$ for every $s^{\star} \leq s' \leq f_{\zeta}(s^{\star})$. So the algorithm will return a partition.

Note that the algorithm does not necessarily return the same signature/partition-property as above $\pi$ that is $\mu$-close to the above partition $A$. The reason for the algorithm to choose a different partition is that there might be another partition of size $s$ with a larger index (which
is of course also $(f_{\zeta},\gamma)$-final) or there might be an $s^* < s$ with the same properties, or there might be other partitions with the same index. However, one can invert the reasoning in the previous paragraph and show that if a $\pi$ is returned then it must be the $\gamma$-signature of an $(f_{\zeta},\gamma)$-final partition.
\end{proof}

\section{Proof of Lemma \ref{lem:info_from_sig}}\label{sec:proof_lemma_2}

\subsection{Preliminary lemmas}

In this subsection we describe some preliminary lemmas that will be used in the next subsection in which we prove Lemma \ref{lem:info_from_sig}.
We start with introducing the Frieze--Kannan regularity lemma \cite{FK-2,FK-1}.
We first state their notion of $\gamma$-regularity.

\begin{definition}[Frieze--Kannan Regularity \cite{FK-1}]\label{def:fk_regularity} Let $G=(V,E)$ be a graph and $A=\{V_{1}\dots,V_{k}\}$ be an equipartition of $V(G)$. For a subset $X \subseteq V$ and $1 \leq i \leq k$ denote $X_i=X\cap V_i$.
We say that $A$ is $\gamma$-Frieze--Kannan-regular if:

\begin{equation}
    d_{\square}^{A}(G) := \max\limits_{ S,T \subseteq V} \frac{1}{n^2}\biggl\lvert \sum_{i,j\in[k]^2}\biggl (d(S_i,T_j)-d_{ij}\biggr )|S_i||T_j|\biggr\rvert<\g
\end{equation}
\end{definition}

Roughly speaking, a partition $A$ is $\gamma$-Frieze--Kannan-regular, or $\gamma$-FK-regular for short, if we can estimate
the number of edges between large sets $S,T$ from the intersection sizes $S \cap V_i$ and $T \cap V_i$.
We will also need the following slightly stronger notion of weak regularity that was introduced in \cite{MS}.

\begin{definition}[Frieze--Kannan Regularity$^{\star}$ \cite{MS}]\label{def:str_fk_regularity}
In the setting of Definition \ref{def:fk_regularity}, we say that $A$ is $\g$-Frieze--Kannan Regular$^{\star}$ if:
\begin{equation}
    d_{\square}^{\star A}(G) := \max\limits_{ S,T \subseteq V} \frac{1}{n^2}\sum_{i,j\in[k]^2}\biggl |d(S_i,T_j)-d_{ij}\biggr ||S_i||T_j|<\g
\end{equation}
\end{definition}

The translation between these two notions will be crucial in Lemma \ref{lem:base_lemma} below.
Suppose $A = \{V_{1},\dots, V_{k}\}$ is an equipartition of $V(G)$. Then an equipartition
$B = \{W_{1}, . . . , W_{\ell}\}$ of $V(G)$ is said to {\em refine} $A$ if each $W_{i}\in B$ is contained in some $V_j \in A$.
The following lemma is proved in \cite{MS} using a simple variant of the original proof of Frieze and Kannan \cite{FK-1}.

\begin{lemma}[Frieze--Kannan Weak Regularity Lemma \cite{FK-1},\cite{MS}]\label{lem:fk_reg_lem} For every $k_{0}$ and $\g >0$ there is $T=T_{\ref{lem:fk_reg_lem}}(k_{0},\g)= k_{0} \cdot 2^{\mathrm{poly}(1/\g)}$ so that the following holds for every graph $G$ on at least $T$ vertices. If $A$ is an equipartition of $V(G)$ into at most $k_{0}$ sets, then there is a refinement $B$ of $A$ into at most $T$ sets such that  $d_{\square}^{\star B}(G)<\g$.
\end{lemma}

Let us now extend the definition of $d_{\square}$ to distance between pairs of weighted graph,
where a weighted graph $R$ is a complete graph, so that every edge $(i,j)$ is assigned a weight $0 \leq R(i,j) \leq 1$.

If $R,R'$ are two weighted graphs on $n$ vertices then we define
\begin{equation}\label{eq1dis}
d_1(R,R')=\frac{1}{n^2}\sum_{i < j}|R(i,j)-R'(i,j)|\;,
\end{equation}
and
\begin{equation}\label{eqsquaredis}
d_{\square}(R,R')=\max_{\alpha,\beta}\frac{1}{n^2}\left|\sum_{i < j}\alpha(i)\beta(j)(R(i,j)-R'(i,j))\right|\;,
\end{equation}
where the maximum is taken over all functions $\alpha,\beta:[n] \rightarrow [0,1]$.

\begin{definition}[$ind(F,R)$]\label{def:ind}
Let $R$ be a weighted graph on $[k]$ and let $\varphi$ be an injective function $\varphi:V(F)\to [k]$. We set
$$
ind_{\varphi}(F,R)=\prod_{i<j\in E(F)}R(\varphi(i),\varphi(j))\prod_{i<j \not \in E(F)}(1-R(\varphi(i),\varphi(j)))
$$
In the case of $\varphi$ not being injective, we define
$$
ind_{\varphi}(F,R)=0
$$
Denoting by $\Phi$ the set of functions from $V(F)$ to $[k]$, we define
\begin{equation}\label{eqindR}
ind(F,R)=\frac{1}{|\Phi|}\sum_{\varphi \in \Phi}ind_{\varphi}(F,R)\;.
\end{equation}
\end{definition}

Note that we can think of a signature $S=(\eta_{i,j})_{1\le i<j\le t}$ as a weighted graph on $t$ vertices.
This means that for a pair of signatures $S,S'$ we can define $d_{1}(S,S')$ and $d_{\square}(S,S')$ as in (\ref{eq1dis}) and (\ref{eqsquaredis}) respectively, and we can also define $ind(F,S)$ as in (\ref{eqindR}). We will need the following lemmas from \cite{HKLLS-1}
\begin{lemma}\label{prop:small_var_under_regularity}
Suppose $R,R'$ are two weighted graphs on $n$ vertices, and
$H$ is a graph on $h$ vertices. Then for any $\g\ge d_{\square}(R,R')$ and $n\ge\frac{2}{\g}$, we have
$|ind(H,R)-ind(H,R')|\le 2h^{2}\cdot \g$
\end{lemma}
\begin{proof}
Let us define a similar notion to $ind(F,R)$, but with respect to non injective functions as well.
For every $\varphi\in\Phi$ (we defined $\Phi$ before (\ref{eqindR})) we set
$$
ind_{\varphi}'(F,R) =\prod_{i<j\in E(F)}R(\varphi(i),\varphi(j))\prod_{i<j \not \in E(F)}(1-R(\varphi(i),\varphi(j)))\;,
$$
and
$$
ind'(F,R)=\frac{1}{|\Phi|}\sum_{\varphi \in \Phi}ind'_{\varphi}(F,R)\;.
$$
Lemma 3.2 in \cite{HKLLS-1} states that if $R,R'$ are two weighted graphs on $n$ vertices, and $H$ is a graph on $h$ vertices, then  $|ind'(H,R)-ind'(H,R')|\le h^{2}\cdot d_{\square}(R,R')$. Now set some $\g\ge d_{\square}(R,R')$ and assume that $n\ge\frac{2}{\g}$.
By Bernoulli's inequality $(n>h)$ we have
$$
|ind'(H,R)-ind(H,R)|\le 1-\frac{(n-h)^h}{n^h}=1-\left(1-\frac{h}{n}\right)^h \leq \frac{h^2}{n}\;.
$$
Thus,
$$
|ind(H,R)-ind(H,R')|\le|ind'(H,R)-ind'(H,R')|+\frac{2h^2}{n}\le h^{2}\cdot d_{\square}(R,R^{\prime}) + h^{2}\g\le 2h^{2}\cdot \g\;,
$$
as desired.
\end{proof}

Given a graph $G$ on $n$ vertices, and an equipartition $A=\{V_1,\dots,V_k\}$, we define the graph $G_A$ on $V(G)$ to be the weighted graph with weights $G_A(u,v)=d(V_i,V_j)$ for every $u \in V_i$ and $v \in V_j$. Let $S_A$ be the $0$-signature of $A$, that is, the weighted graph on $k$ vertices with $S(i,j)=d(V_i,V_j)$. Observe that if $k$ divides $n$ (so all sets of $A$ are of equal size) then $ind(H,G_A)$ is almost the same as $ind(H,S_A)$. It is not hard to see that for general equipartitions these quantities do not differ my much.

\begin{lemma}\label{lem:blow_up_graph}
Given a graph $G$ on $n$ vertices, and an equipartition $A=\{V_1,\dots,V_k\}$, let $G_A$ and $S_A$ be defined as above. Then $|ind(H,G_A)-ind(H,S_A)|\le \frac{2h^2}{k}+\frac{2kh}{n}$ for every graph $H$ on $h$ vertices.
\end{lemma}

\begin{proof}
We use the definition of $ind'(F,R)$ introduced in the proof of Lemma \ref{prop:small_var_under_regularity}.
By Inequality (5) in \cite{HKLLS-1} we have $|ind'(H,G_A)-ind'(H,S_A)|\le\frac{2kh}{n}$.
Note that with the definitions above, as in the proof of Lemma \ref{prop:small_var_under_regularity},
$$
|ind'(H,G_A)-ind(H,G_A)|\le \frac{h^2}{n}\ \text{and }|ind'(H,S_A)-ind(H,S_A)|\le \frac{h^2}{k}\;.
$$
Thus, we have
$$|ind(H,G_A)-ind(H,S_A)|\le \frac{2h^2}{k}+\frac{2kh}{n}\;,$$
as desired.
\end{proof}

We now combine the above facts to conclude that a signature of a $\gamma$-FK-partition of a graph gives a good approximation of $ind(H,G)$.

\begin{lemma}\label{lem:sample_with_sig} For every $h,k$ and $\delta>0$ there are
$$
\g=\g_{\ref{lem:sample_with_sig}}(h,\delta)=\mathrm{poly}(\delta/h),~~r=r_{\ref{lem:sample_with_sig}}(h,\delta)=\mathrm{poly}(h/\delta),~~N=N_{\ref{lem:sample_with_sig}}(h,k,\delta)=\mathrm{poly}(hk/\delta)\;,
$$
so that if $G$ is a graph on at least $N$ vertices,
and $A$ is a $\g$-FK-regular partition of $G$ with at least $r$ and up to $k$ parts, then for every $\g$-signature $S$ of $A$, we have $|ind(H,G)-ind(H,S)| \le \delta$ for every $H$ on $h$ vertices.
\end{lemma}

\begin{proof}
We set $\gamma=\frac{\delta}{8h^{2}},r\ge 2/\gamma$ and $ N= \max\{2/\g,\frac{2k^2}{h},\frac{16kh}{\delta}\}$.
Let $G_A$ and $S_A$ be as defined before Lemma \ref{lem:blow_up_graph}.
If we view $G$ as a weighted $0/1$ graph (so $G(x,y)=1$ if and only if $(x,y) \in E(G)$) then we have
\begin{align}
d_{\square}(G,G_A)&=\max_{\alpha,\beta}\frac{1}{n^2}\left|\sum_{x,y \in V(G)}\alpha(x)\beta(y)(G(x,y)-G_A(x,y))\right|\nonumber\\
&=\max\limits_{S,T \subseteq V} \frac{1}{n^{2}}\biggl\lvert\sum_{i,j\in[k]^{2}}\sum_{x\in S_i,y\in T_j}(G(x,y)-d_{ij})\biggl\lvert\nonumber\\
&=\max\limits_{S,T \subseteq V} \frac{1}{n^{2}}\biggl\lvert\sum_{i,j\in[k]^{2}}e(S_i,T_j)-d_{ij}|S_i||T_j|\biggl\lvert\nonumber\\
&=\max\limits_{S,T \subseteq V} \frac{1}{n^{2}}\biggl\lvert\sum_{i,j\in[k]^{2}}(d(S_i,T_j)-d_{ij})|S_i||T_j|\biggl\lvert\nonumber\\
&=d_{\square}^{A}(G)\leq \g\nonumber
\end{align}
where in the second equality we used the fact that the maximum is always achieved by Boolean\footnote{Indeed,
assume without loss of generality that the maximum is positive. Then we can round to $1$ every $\alpha(x)$ if increasing it increases the outcome.
We can round to $0$ all the rest. We can then do the same rounding process with respect to $\beta$.} valued $\alpha,\beta$.
We may thus infer from Lemma \ref{prop:small_var_under_regularity} (applied with the above defined $\g$; note that $N\ge 2/\g$) that
$|ind(H,G)-ind(H,G_A)|\le 2h^{2} \cdot\g$ for every $H$.
By Lemma \ref{lem:blow_up_graph}, for every $H$ on $h$ vertices, we have $|ind(H,G_A)-ind(H,S_A)|\le \frac{2h^2}{k} + \frac{2kh}{n}$. By our choice of $r,\g$ and $N$, we have that $|ind(H,G_A)-ind(H,S_A)|\le \delta/4$.
Hence by the triangle inequality, and our choice of $\gamma$, we have
$|ind(H,G)-ind(H,S_A)| < \delta/2$. Finally, since $S$ is a $\gamma$-signature of $A$ we have
$d_{\square}(S,S_A) \leq d_1(S,S_A) \leq \gamma$ so by another application of Lemma \ref{prop:small_var_under_regularity} (again with $\g$) we also
have $|ind(H,S)-ind(H,S_A)|\le\delta/2$. Hence, by another application of the triangle inequality we deduce that $|ind(H,G)-ind(H,S)| \le \delta$
thus completing the proof.
\end{proof}

\begin{definition}[Extension]\label{def:extension}
Given a signature $S=(\eta_{ij})_{1\le i<j\le t}$ of an equipartition $A$, and a refinement $B=\{W_{1},\dots,W_{s}\}$ of $A$, the extension of $S$ to $B$ is the sequence $S^{\prime}=(\eta^{\prime}_{ij})_{1\le i<j\le s}$ defined as $\eta_{i,j}^{\prime}=\eta_{k,l}$ if there exist $k\not=l$ such that $W_{i}\subseteq V_{k}$ and $W_{j}\subseteq V_{l}$, and setting $\eta_{i,j}^{\prime}=0$ if $W_{i}$ and $W_{j}$ are both subsets of the same $V_k$.
\end{definition}

\begin{claim}\label{prop:dist_sig}
For every $\eps$ and $s$ there exists $r=r_{\ref{prop:dist_sig}}(\eps)=\mathrm{poly}(1/\eps)$ and $N=N_{\ref{prop:dist_sig}}(\eps,s)=\mathrm{poly}(s/\eps)$ so that the following holds for every pair of graphs $G,G^{\prime}$ on the same set of $n\ge N$ vertices. If $G,G'$
are $\alpha$-close and $S,S^{\prime}$ are $\g,\g^{\prime}$-signatures of $G,G'$ respectively, of the same equipartition $A$ of the vertex set of $G,G^{\prime}$ into $s\ge r$ sets, then $d_1(S,S^{\prime}) \leq \alpha+\eps+2(\g+\g^{\prime})$.
\end{claim}

\begin{proof}
We set $r=2/\eps$. Let $S_A,S'_A$ be the $0$-signatures of $A$ with respect to $G,G'$.
Then assuming $n>N_{\ref{prop:dist_sig}}(\eps,s)=\mathrm{poly}(s/\eps)$ is large enough we clearly have
$d_1(S_A,S'_A) \leq \alpha+\eps$.
Since $S$ is a $\gamma$-signature of $A$ we have (by definition) $d_1(S,S_A) \leq 2\g$, and by the same reasoning
we have $d_1(S',S'_A) \leq 2\g'$. Hence, by the triangle inequality we have $d_1(S,S^{\prime}) \leq \alpha+\eps+2(\g+\g^{\prime})$.
\end{proof}

\begin{lemma}\label{lem:base_lemma}
For every $\eps$ and $t$ there exists $\g=\g_{\ref{lem:base_lemma}}(\eps)=\mathrm{poly}(\eps)$ and $N=N_{\ref{lem:base_lemma}}(t,\eps)=\mathrm{poly}(t/\eps)$, so that for every graph $G$ on $n\ge N$ vertices, if $S$ is a $\g$-signature of a $\g$-FK-regular$^\star$ partition $A$ of $G$ with $t$ sets, then for every signature $S^{\prime}$ satisfying $d_1(S,S') \leq \delta$ for some $\delta$, there is a graph $G^{\prime}$ that is $(\delta+\eps)$-close to $G$, so that $A$ is an $\eps$-FK-regular partition of $G^{\prime}$, and $S^{\prime}$ is an $\eps$-signature of $A$.
\end{lemma}

\begin{proof}
The idea is very simple; we randomly modify $G$ so that the densities will be those of $S'$.
However, showing that $A$ will be an $\eps$-FK-regular partition of $G'$ will require a subtle argument that will employ
the fact that $A$ is an $\g$-FK-regular$^\star$ partition of $G$.

We set $\g=\g_{\ref{lem:base_lemma}}(\eps)=\eps/6$ and $N=N_{\ref{lem:base_lemma}}(t,\eps)=960t^{3}/\g^{5}$.
Given $G$, $A=\{V_{1},\dots V_{t}\}$, $S=(\eta_{ij})_{1\le i<j\le t}$ and $S^{\prime}=(\eta^{\prime}_{ij})_{1\le i<j\le t}$ as in the statement of the lemma, we obtain $G^{\prime}$ from $G$ using the following process, in which all random choices are done independently:
\begin{itemize}
    \item For every $i$, the edges within $V_{i}$ are unchanged.
    \item For $i<j$ such that $\eta_{i,j}^{\prime}<d(V_{i},V_{j})$, every edge of $G$ between $V_{i}$ and $V_{j}$ is removed with probability $1-\frac{\eta_{i,j}^{\prime}}{d(V_{i},V_{j})}$.
    \item  For $i<j$ such that $\eta_{i,j}^{\prime}>d(V_{i},V_{j})$, every vertex pair of $G$ between $V_{i}$ and $V_{j}$ that is not an edge, becomes an edge with probability $1-\frac{1-\eta_{i,j}^{\prime}}{1-d(V_{i},V_{j})}$.
\end{itemize}

In what follows we use $e'(X,Y)$ to denote the number of edges in $G'$ between $X$ and $Y$ and $d'(X,Y)$ to denote the densities between
these sets in $G'$. Note that the way we generate $G'$ guarantees that for every $i<j$ we have
\begin{equation}\label{eqexpectation}
\mathbb{E}[d'(V_i,V_j)]=\eta'_{i,j}\;.
\end{equation}

We first prove that with probability at least $1/2$, for every $i<j$ and every $X\subseteq V_i,Y\subseteq V_j$ satisfying $|X| \ge 2\g|V_i|,|Y|\ge2\g|V_j|$, we have
\begin{equation}\label{inequ1}
|e'(X,Y)-\mathbb{E}[e'(X,Y)]|\le \frac{\g}{4}|X||Y|.
\end{equation}
Note also that it is equivalent to
\begin{equation}\label{inequ2}
|d'(X,Y)-\mathbb{E}[d'(X,Y)]|\le \g/4\;.
\end{equation}
It suffices to show that for fixed $i < j$ and subsets $X \subseteq V_i,Y\subseteq V_j$ as above we have
$$
\Pr\left[|e'(X,Y)-\mathbb{E}[e'(X,Y)]|\ge \frac{\g}{4}|X||Y|\right]<\frac{2^{-3n/t}}{4t^2}< \frac{2^{-2(n/t+1)}}{4t^2} \;,
$$
since we could then conclude by taking a union bound over all $i <j$ and choices of $X,Y$.
So from this point we fix $i,j,X,Y$. We first treat the case $d(V_i,V_j)>\eta'_{ij}$.
If $d(X,Y)<\g/4$ then both $e'(X,Y),\mathbb{E}[e'(X,Y)] \leq \frac{\gamma}{4}|X||Y|$
so (\ref{inequ1}) holds with probability 1.
Assume now that $d(X,Y)\ge\g/4$.
By Lemma \ref{lem:chern}, with $m=d(X,Y)|X||Y|\ge \g^3 (n/t-1)^2 \ge \g^3 \frac{n^2}{4t^2}$ and $\theta=\frac{\g}{4d(X,Y)}$ we get
$$
\Pr[|e'(X,Y)-\mathbb{E}[e'(X,Y)]|> \frac{\g}{4}|X||Y|]<2e^{-2\theta^2m}<2e^{-2(\frac{\g}{4d(X,Y)})^2\cdot \frac{\g^3 n^2}{4t^2}}
\le 2e^{-\frac{1}{32}(\frac{\g^5 n^2}{t^2})}\le  \frac{2^{-3n/t}}{4t^2}\;,
$$
where in the last inequality we used our assumption on $n$.
We now assume that $d(V_i,V_j)<\eta'_{ij}$.
If $d(X,Y)>1-\g/4$ then again (\ref{inequ1}) holds with probability 1.
Assume now that $d(X,Y)\le 1-\g/4$.
As above, by Lemma \ref{lem:chern}, with $m=(1-d(X,Y))|X||Y|\ge \g^3 (n/t-1)^2 \ge \g^3 \frac{n^2}{4t^2}$ and $\theta=\frac{\g}{4(1-d(X,Y))}$ we get
$$
\Pr[|e'(X,Y)-\mathbb{E}[e'(X,Y)]|> \frac{\g}{4}|X||Y|]<2e^{-2\theta^2m}<2e^{-2(\frac{\g}{4(1-d(X,Y))})^2\cdot \frac{\g^3 n^2}{4t^2}}
\le 2e^{-\frac{1}{32}(\frac{\g^5 n^2}{t^2})}\le   \frac{2^{-3n/t}}{4t^2}
$$
Since the statement clearly holds for the case of $d(V_i,V_j)=\eta'_{ij}$ (in this case we do nothing),
we have thus proved that (\ref{inequ2}) holds with probability at least $1/2$ for all $i,j$ and every $X\subseteq V_i,Y\subseteq V_j$ satisfying $|X| \ge 2\g|V_i|,|Y|\ge2\g|V_j|$. We will now prove that this fact implies all the assertions of the lemma.

We first observe that using $X=V_i$ and $Y=V_j$ in (\ref{inequ2}) we see that for every $i<j$ we have
\begin{equation}\label{eqdensities3}
|d'(V_i,V_j)-\mathbb{E}[d'(V_i,V_j)]|=|d'(V_i,V_j)-\eta'_{i,j}| \leq \g/4\;,
\end{equation}
where the first equality is (\ref{eqexpectation}).

We now claim that for every $i<j$ and $X\subseteq V_i,Y\subseteq V_j$, we have
\begin{equation}\label{ineq2}
    |\mathbb{E}[d'(X,Y)]-\eta'_{i,j}|\le |d(X,Y)-d_{i,j}|
\end{equation}
where we use $d_{i,j}=d(V_i,V_j)$.
This clearly holds if $d_{i,j}=\eta'_{i,j}$, so assume first
that $d_{i,j}>\eta'_{i,j}$. Setting $q_{ij}=\eta'_{i,j}/d_{i,j}$, we see that every edge between $V_i$ and $V_j$ is kept in $G'$ with probability $q_{i,j}$, hence
\begin{align*}
|\mathbb{E}[d'(X,Y)]-\eta'_{i,j}|=|q_{ij}\cdot d(X,Y)-\eta'_{i,j}|=q_{ij}|d(X,Y)-d_{i,j}|\leq |d(X,Y)-d_{i,j}|\;.
\end{align*}

Similarly, if $d_{i,j}<\eta'_{i,j}$, then setting $q'_{ij}=1-p'_{ij}=(1-\eta'_{i,j})/(1-d_{i,j})$
we see that every edge missing between
$V_i$ and $V_j$ is added to $G'$ with probability $p'_{i,j}$, hence

\begin{align}
    |\mathbb{E}[d'(X,Y)]-\eta'_{ij}|&=|d(X,Y)+p'_{ij}(1-d(X,Y))-\eta'_{ij}|\nonumber\\
    &=|q'_{ij}\cdot d(X,Y)+p'_{ij}-\eta'_{i,j}|\nonumber\\
    &=|q'_{ij}\cdot d(X,Y)-q'_{ij}+(1-\eta'_{i,j})|\nonumber\\
    &=q'_{ij}|d(X,Y)-d_{i,j}| \le |d(X,Y)-d_{i,j}| \; .\nonumber
\end{align}

Fix now a pair of sets $S,T$ and let $L$ be the set of pairs $i,j$ for which $|S_i| \geq 2\gamma |V_i|$ and $|T_j| \geq 2\gamma |V_j|$. Then
\begin{align}
    \biggl|\sum_{i,j\in[k]^2}(d'(S_i,T_j)-d'_{ij})\cdot |S_i||T_j|\biggr|&\le 4\g n^2+ \biggl|\sum_{i,j\in L}(d'(S_i,T_j)-d'_{ij})\cdot |S_i||T_j|\biggr|\nonumber\\
    &\le 5\g n^2+\biggl|\sum_{i,j\in L}(\mathbb{E}[d'(S_i,T_j)]-\eta'_{i,j})\cdot |S_i||T_j|\biggr|\nonumber\\
    &\le 5\g n^2+ \sum_{i,j\in L}\biggl |(\mathbb{E}[d'(S_i,T_j)]-\eta'_{i,j})\biggr |\cdot |S_i||T_j|\nonumber\\
    &\le 5\g n^2+ \sum_{i,j\in L}\biggl |d(S_i,T_j)-d_{ij})\biggr |\cdot |S_i||T_j|\nonumber\\
    &\le 5\g n^2+ \sum_{i,j\in [k]^2}\biggl |d(S_i,T_j)-d_{ij}\biggr |\cdot |S_i||T_j|\nonumber\\
    &\le 6\g n^2 = \eps n^2\nonumber\;,
\end{align}
where the first inequality holds by the definition of $L$, the second inequality holds due to (\ref{eqdensities3}) and (\ref{inequ2}) (applied to $(S_i,T_j)$), the third inequality is the triangle inequality, the fourth inequality is (\ref{ineq2}) and the sixth inequality is the assumption that $d^{\star A}_{\square}(G)\leq \gamma$.
Since the above holds for every $S,T$ we deduce that $d^{A}_{\square}(G')\leq \eps$ so $A$ is indeed an $\eps$-FK-regular partition of $G'$.

Since $\gamma < \eps$, inequality (\ref{eqdensities3}) implies that $S'$ is an $\eps$-signature of $A$ with respect to $G'$, establishing the third assertion of the lemma.
We also deduce from (\ref{eqdensities3}) that
$$
\frac{1}{k^2}\sum_{i,j\in[k]^{2}}|d^{\prime}(V_{i},V_{j})-\eta'_{i,j}|\leq \gamma \leq \eps/2\;.
$$
By the lemma's assumption we also have
$$
\frac{1}{k^2}\sum_{i,j\in[k]^{2}}|\eta_{i,j}-\eta'_{i,j}|=d(S,S') \leq \delta\;.
$$
By the lemma's assumption, $S$ is a $\gamma$-signature of $A$ with respect to $G$, implying that
$$
\frac{1}{k^2}\sum_{i,j\in[k]^{2}}|d(V_{i},V_{j})-\eta_{i,j}| \leq 2\gamma \leq \eps/2
$$
Finally, since
$$
|d^{\prime}(V_{i},V_{j})-d(V_{i},V_{j})|\le|d^{\prime}(V_{i},V_{j})-\eta'_{i,j}|+|\eta_{i,j}-\eta'_{i,j}|+|\eta_{i,j}-d(V_{i},V_{j})|,
$$
we infer that
$$
\frac{n^2}{k^2}\sum_{i,j\in[k]^{2}} |d^{\prime}(V_{i},V_{j})-d(V_{i},V_{j})| \leq (\delta + \eps)n^2\;.
$$
Since the left hand side above is the precise number of edge modifications we made when changing $G$ to $G'$, we deduce that
$G'$ is $(\delta+\eps)$-close to $G$, establishing the first assertion of the lemma.
\end{proof}

We will also need the following lemmas.

\begin{lemma}(\cite{AFKS} Lemma 3.7)\label{lem:add_index}
For every $\eps,t$ there exists $\g=\g_{\ref{lem:add_index}}(\eps)=\mathrm{poly}(\eps)$ and $N=N_{\ref{lem:add_index}}(t,\eps)=\mathrm{poly}(t/\eps)$ satisfying the following.
Assume $A$ is an equipartition into $s$ sets of a graph $G$ with $n\ge N$ vertices, and that $B$ is a refinement of $A$ into at most $t$ sets. Assume further that $S$ is any $\g$-signature of $A$, and that $T$ is its extension to $B$.
If $B$ satisfies $ind(B)\le ind(A)+\g$, then $T$ is an $\eps$-signature for $B$.
\end{lemma}

\begin{lemma}(\cite{FN} Lemma 6.6)\label{lem:close_ind}
For every $\eps,t$ there exists $N=N_{\ref{lem:close_ind}}(t,\eps)=\mathrm{poly}(t/\eps)$ so that for every equipartition $A$ of $G$ with $n\ge N$ vertices into $s$ sets, and every refinement $B$ of $A$ into at most $t$ sets, $ind(B)\ge ind(A)-\eps$.
\end{lemma}

The next observation is implicit in the proof of the Frieze--Kannan Regularity Lemma (i.e. Lemma \ref{lem:fk_reg_lem}).
The main step of the proof involves showing that if $A$ is an equipartition of $G$ into $t$ parts and $A$ is not $\eps$-FK-regular$^{\star}$, then $A$ has a refinement $B$ into $k \le 16t/\eps^4$ sets so that $ind(B) \geq ind(A) + \frac{\eps^4}{2}$ (see, e.g., the proof of Theorem 1.1 in \cite{RSsurvey} and the proof of Theorem 6 in \cite{MS}).


\begin{lemma}\label{lem:robust_to_regular}
For every $\eps>0$ there exists $\gamma=\g_{\ref{lem:robust_to_regular}}(\eps)=\mathrm{poly}(\eps)$ and $f=f_{\ref{lem:robust_to_regular}}^{(\epsilon)}:\mathbb{N}\to\mathbb{N}$ satisfying $f(x)=\mathrm{poly}(1/\eps)\cdot x$ and such that every $(f,\g)$-final partition of a graph is also $\eps$-FK-regular$^{\star}$.
\end{lemma}

\begin{lemma}\label{lem:resp_reg}
For every $s$ and $\eps>0$ there are $\g=\g_{\ref{lem:resp_reg}}(\eps)$, $T=T_{\ref{lem:resp_reg}}(s,\eps)$, $f=f_{\ref{lem:resp_reg}}^{(\eps)}$ and $N=N_{\ref{lem:resp_reg}}(\eps,s)$ so that
$$
\g=\mathrm{poly}(\eps),~~T=s \cdot 2^{\mathrm{poly}(1/\eps)},~~ f(x)=x\cdot2^{\mathrm{poly}(1/\eps)},~~N=\mathrm{poly}(s) \cdot 2^{\mathrm{poly}(1/\eps)}
$$
and the following holds.
Suppose $G$ has at least $N$ vertices and $A$ is an $(f,\g)$-final partition of $G$ into at most $s$ sets and that $S$ is a $\g$-signature of $A$.
Then for every $G^{\prime}$ on the same vertex set of $G$, there exists a refinement $A'$ of $A$ into $t\le T$ sets so that
\begin{itemize}
\item[$(i)$] $A'$ is an $\eps$-FK-regular$^{\star}$ partition of $G'$.
\item[$(ii)$] Every refinement $A''$ of $A$ with $t\le T$ sets (and in particular $A'$), is an $\eps$-FK-regular$^{\star}$ partition of $G$.
\item[$(iii)$] For every refinement $A''$ of $A$ with $t\le T$ sets, the extension $S''$ of $S$ (in the sense of Definition \ref{def:extension}) with respect to $A''$ is an $\eps$-signature of $A''$ with respect to $G$ (note that $A'$ is such an $A''$).
\end{itemize}
\end{lemma}

\begin{proof}
Given $s$ and $\eps$ we define
$$
\g=\min\{\frac{1}{2}\g_{\ref{lem:robust_to_regular}}(\eps),\g_{\ref{lem:add_index}}(\eps)\}=\mathrm{poly}(\eps)\;,~~T=T_{\ref{lem:fk_reg_lem}}(s,\eps)=s \cdot 2^{\mathrm{poly}(1/\eps)}\;,
$$
$$
f(x)=f_{\ref{lem:robust_to_regular}}^{(\eps)}(T_{\ref{lem:fk_reg_lem}}(x,\eps))=2^{\mathrm{poly}(1/\eps)}\cdot x \cdot 2^{\mathrm{poly}(1/\eps)}=x \cdot  2^{\mathrm{poly}(1/\eps)}\;,
$$
and
$$
N=\max\{T,N_{\ref{lem:close_ind}}(T,\g),N_{\ref{lem:add_index}}(T,\eps)\}=\mathrm{poly}(s) \cdot 2^{\mathrm{poly}(1/\eps)}\;.
$$
Given an $(f,\gamma)$-final partition $A$, and assuming that $N\ge N_{\ref{lem:fk_reg_lem}}(s,\eps)$, Lemma \ref{lem:fk_reg_lem} produces a refinement $A'$ of $A$ that partitions $G^{\prime}$ into at most $T_{\ref{lem:fk_reg_lem}}(s,\eps)$ sets and is $\eps$-FK-regular$^{\star}$ with respect to $G^{\prime}$.
It remains to prove Items $(ii)$-$(iii)$.
Lemma \ref{lem:close_ind} asserts that with respect to $G$, every partition $A''$ that refines $A$ with at most $T$ sets, satisfies $ind(A'')\ge ind(A)-\g\ge ind(A)-\frac{1}{2}\g_{\ref{lem:robust_to_regular}}(\eps)$ (since $N>N_{\ref{lem:close_ind}}(T,\g)$).
This implies that in $G$, the partition $A''$ is $(f_{\ref{lem:robust_to_regular}}^{(\eps)},\g_{\ref{lem:robust_to_regular}}(\eps))$-final. Indeed, if there was a partition $C$ with at most $f_{\ref{lem:robust_to_regular}}^{(\eps)}(T_{\ref{lem:fk_reg_lem}}(x,\eps))$ sets for which $ind(C)>ind(A'')+\g_{\ref{lem:robust_to_regular}}(\eps)$ in $G$, then this would imply that $ind(C)>ind(A)+\g_{\ref{lem:robust_to_regular}}(\eps)-\frac{1}{2}\g_{\ref{lem:robust_to_regular}}(\eps)$ contradicting the $(f,\gamma)$-finality of $A$. We may thus infer via Lemma \ref{lem:robust_to_regular}, that $A''$ is also an $\eps$-FK-regular$^{\star}$ partition with respect to $G$, establishing item $(ii)$.
Finally (pun intended), the $(f,\gamma)$-finality of $A$ in $G$ ensures that for every partition $A''$ of $G$ into at most $T$ sets (and in particular for every refinement of $A$ into this many sets)
we have $ind(A'') \leq ind(A)+\g_{\ref{lem:add_index}}(\eps)$. Hence, by Lemma \ref{lem:add_index}, the extension $S''$ of $S$ to such an $A''$ is an $\eps$-signature of $A''$ with respect to $G$, establishing item $(iii)$.
\end{proof}

\subsection{Proof of Lemma \ref{lem:info_from_sig}}
Given $h,\eps$ and $\delta$ we first choose
$$
\g_{0}=\min\{\eps/10,\g_{\ref{lem:sample_with_sig}}(h,\delta/6),\g_{\ref{lem:base_lemma}}(\min\{\eps/2,\g_{\ref{lem:sample_with_sig}}(h,\delta/6)\})\}=\mathrm{poly}(\eps\delta/h)\;,
$$
and then define
$$
\g=\g_{\ref{lem:resp_reg}}(\g_{0})=\mathrm{poly}(\eps\delta/h),~~~ s=\max\{r_{\ref{lem:sample_with_sig}}(h,\delta/6),r_{\ref{prop:dist_sig}}(\eps/10),20h^2/\delta\}=\mathrm{poly}(h/\eps\delta)\;,
$$
$$
f(x)=f_{\ref{lem:resp_reg}}^{(\g_{0})}(x)=x\cdot2^{{\mathrm{poly}(1/\g_{0})}}
=x\cdot2^{\mathrm{poly}(h/\eps\delta)}\;,
$$
to be the constants and function in the statement of Lemma \ref{lem:info_from_sig}, noting that they satisfy the guarantees of that lemma.
Given $t$ as in the statement of Lemma \ref{lem:info_from_sig}, we set
$$
T=T_{\ref{lem:resp_reg}}(t,\g_{0})
$$
and define
\begin{align}
N=\max\{N_{\ref{lem:sample_with_sig}}(h,T,\delta/6),N_{\ref{prop:dist_sig}}(\eps/10,T),N_{\ref{lem:resp_reg}}(\g_0,s),N_{\ref{lem:base_lemma}}(t,\g_0)\}
 = \mathrm{poly}(t) \cdot 2^{\mathrm{poly}(h/\eps\delta)}\nonumber\;,
\end{align}
to be the constant in Lemma \ref{lem:info_from_sig}.

Given a family of graphs ${\cal H}$ on at most $h$ vertices, we define a family of signatures as follows
$$
\mathcal{C}_{\delta,{\cal H},T}=\{C:|C|\leq T ~\mbox{and}~ind(H,C)\leq \delta/2~\mbox{for every}~ H \in {\cal H}\}\;.
$$
In order for $\mathcal{C}_{\delta,{\cal H},T}$ to be finite, we only put in it signatures $C$ with edge weights $\eta_{i,j}$ that are integer multiples of $\beta=\min\{\eps/10,\delta/10h^2\}$.
Intuitively, this is the set of signatures ``certifying'' (hence $C$) that a graph with that signature is close to being induced ${\cal H}$-free.
We also define $\mathcal{S}_{T}$ to be the set of all signatures on up to $T$ parts, that are extensions\footnote{Note that strictly speaking,
an extension per Definition \ref{def:extension} must be relative to a partition $A$ and its refinement $B$, while here we only have the signature $S$. So what we mean here is that if one takes some graph that has a partition $A$ whose $0$-signature is $S$, then
$\mathcal{S}_{T}$ is the family of all signatures that one obtains by taking all refinements of $A$ into at most $T$ sets,
and then taking the extension of $S$ to these refinements. Of course we do not need any graph in order to produce ${\cal S}_T$; we just
break the ``parts'' of $S$ into a total of at most $T$ new ``parts'', and then define the densities $\eta'_{i,j}$ between the new vertices as in Definition \ref{def:extension}.} of $S$. Intuitively, these are the signatures one can obtain by refining
$A$ into at most $T$ sets (recall that the crucial point is that the algorithm only has access to $S$ and not to $G$).

Suppose now that we are given a $\gamma$-signature $S$ of some $(f,\gamma)$-final (with the above defined $f,\gamma$) partition $A$ of a graph $G$, so that $S$ has $t\geq s$ parts and $G$ has at least $N$ vertices.
The algorithm checks if there are $S'\in \mathcal{S}_{T}$ and $C\in {\cal C}_{\delta,{\cal H},T}$ satisfying $d_1(S',C)\le \alpha-\frac{\eps}{2}$. If there is such a pair, the algorithm says that case $(i)$ holds, otherwise it says that case $(ii)$ holds. We now prove the correctness of the algorithm.

\paragraph{Proof of first direction:} Suppose there is a graph $G^{\prime}$ which is $(\alpha-\eps)$-close to $G$, and satisfies
$ind(H,G')=0$ for every $H \in {\cal H}$. We will show that the algorithm will declare that case $(i)$ holds.

Recall that $A$ is an $(f,\g)$-final partition of $G$ into $t\ge s$ sets and that $S$ is a $\g$-signature of $A$. By Lemma \ref{lem:resp_reg}, there exists a refinement $A'$ of $A$ into at most $T$ sets so that $A'$ is $\g_{0}$-FK-regular$^{\star}$ for both $G$ {\bf and} $G^{\prime}$. Moreover, denoting by $S'$ the corresponding extension of $S$ to $A'$, we have that $S'$ is a $\g_{0}$-signature of $A'$ with respect to $G$. Note that $S'\in \mathcal{S}_{T}$.
By the choice of $\g_{0}$, this implies that $A'$ is $\g_{\ref{lem:sample_with_sig}}(h,\delta/6)$-FK-regular$^{\star}$ for both $G$ and $G^{\prime}$, and that $S'$ is a $\frac{1}{10}\eps$-signature of $A'$ with respect to $G$.
Let $C'$ be the $0$-signature of $A'$ over $G^{\prime}$. Lemma \ref{lem:sample_with_sig} (using $A'$ and $G^{\prime}$) implies that $|ind(H,G')-ind(H,C')| \leq \delta/6$ for all $H \in {\cal H}$.
Thus $ind(H,C') \le \delta/6$ for all $H \in {\cal H}$.
Clearly there is a signature $C$ of size $C'$ so that all of $C$'s weights are constant multiples of $\beta$
and $d_1(C',C) \leq \beta$. Since $d_{\square}(C',C) \leq d_1(C',C) \leq \delta/10h^2$
we infer from Lemma \ref{prop:small_var_under_regularity} (applied on $\frac{\delta}{10h^2}$, as $s\ge\frac{20h^2}{\delta}$) that $ind(H,C)\le \delta/6 +\delta/5 < \delta/2$ for all $H \in {\cal H}$, so
$C\in \mathcal{C}_{\delta, {\cal H},T}$.
In addition, by Claim \ref{prop:dist_sig} (since $A'$ has at least $r_{\ref{prop:dist_sig}}(\eps/10)$ parts and assuming that $n$ is large enough), we infer that $d_1(S',C)\leq \alpha-\frac{\eps}{2}$ (since $G$ and $G^{\prime}$ are $(\alpha-\eps)$-close and $d_1(C,C')\leq \eps/10$). Thus, $S'$ and $C$ provide a witness that the algorithm will indeed declare that case $(i)$ holds.

\paragraph{Proof of second direction:} Suppose the algorithm declares that case $(i)$ holds. We show that in this case
there is a graph $G'$, which is $\alpha$-close to $G$, and satisfies $ind(H,G')<\delta$ for all $H\in\mathcal{H}$

Indeed, if the algorithm declared that case $(i)$ holds then there are signatures $S'\in \mathcal{S}_{T}$ and $C \in \mathcal{C}_{\delta,{\cal H},T}$ satisfying $d_1(S',C) \le \alpha-\frac{\eps}{2}$.
As $S'\in\mathcal{S}_{T}$, there is a refinement $A'$ of $A$, so that $S'$ is the extension of $S$ according to $A'$.
Lemma \ref{lem:resp_reg} (regarding $A'$ as a possible refinement of $A$ with respect to $G$) asserts that $S'$ is a $\g_{0}$-signature of $A'$ (with respect to $G$), which by the choice of $\g_{0}$ means that it is a $\g_{\ref{lem:base_lemma}}(\min\{\frac{\eps}{2},\g_{\ref{lem:sample_with_sig}}(h,\delta/6)\})$-signature for $A'$ with respect to $G$.
Now, Lemma \ref{lem:base_lemma} (applied with $A'$ as the $\g_{0}$-FK-regular$^{\star}$ partition of $G$, and with $S'$ as $S$ and $C$ as $S'$) implies that there is a graph $G^{\prime}$ that is $(\alpha-\frac{\eps}{2}+\frac{\eps}{2})$-close to $G$, namely $\alpha$-close to $G$, and for which $C$ is a $\g_{\ref{lem:sample_with_sig}}(h,\delta/6)$-signature of $A'$, which in turn is $\g_{\ref{lem:sample_with_sig}}(h,\delta/6)$-FK-regular over $G^{\prime}$ . Lemma \ref{lem:sample_with_sig} implies that $|ind(H,G')-ind(H,C)|\leq \delta/6$ for all $H \in  {\cal H}$. Thus, $ind(H,G')<\delta/2+\delta/6<\delta$ for all $H\in\mathcal{H}$ as required.
Hence we have found the required $G'$.

\section{Proof of Theorem \ref{thm:main_gen}}\label{sec:conc}
The proof of Theorem \ref{thm:main_gen} is very similar to that of Theorem \ref{thm:main}.
In order to assist the reader who is already familiar with the proof of Theorem \ref{thm:main}, we mention in several places where certain
lemmas are analogous to lemmas we introduced in one of the previous sections.
The idea is the following: by a theorem of Goldreich and Trevisan \cite{GT}, every testable property is testable
by a canonical tester, which samples a set of vertices of size $q=q_{\cal P}(\varepsilon)$ and accepts/rejects based on the
graph induced by these $q$ vertices. Hence the acceptance/rejection of the algorithm only depends on the number of induced
copies in $G$ of graphs on $q$ vertices. Hence, turning a graph into a graph satisfying ${\cal P}$ is equivalent
to turning it into a graph with a certain number of copies of certain graphs on $q$ vertices. As evident, this is very
similar to the case of Theorem \ref{thm:main} where we wanted to have a very small number of copies of graphs not in ${\cal P}$.
The reason why there is an additional exponential factor is that we need to control the number of induced copies of {\em all} graphs
on $q$ vertices.

We now state the key lemmas, which are variants of lemmas we used in the proof of Theorem \ref{thm:main}.

\begin{definition}\label{def:var_dist}
Given two distributions $\mu$ and $\nu$ over a finite family $\mathcal{H}$ of combinatorial structures, their variation distance is defined as:
$|\mu-\nu|=\frac{1}{2}\sum_{H\in\mathcal{H}}|\Pr_{\mu}(H)-\Pr_{\nu}(H)|$
\end{definition}

\begin{lemma}\label{lem:sub_var_dist}
If two distributions $\mu$ and $\nu$ over a finite family  $\mathcal{H}$ of combinatorial structures satisfy $|\mu-\nu|\le\delta$ , then for any set $A\subset\mathcal{H}$ we have
$|\Pr_{\mu}(A)-\Pr_{\nu}(A)|\le\delta$
\end{lemma}

\begin{lemma}\label{lem:graph_var_dist} Suppose that $\mu$ and $\nu$ are two probability distributions over graphs with set of vertices $\{v_{1},\dots,v_{q}\}$, where each edge $v_{i}v_{j}$ is independently chosen to be an edge with probability $\mu_{i,j}$ and $\nu_{i,j}$ respectively. If $|\mu_{i,j}-\nu_{i,j}|\le\epsilon/\binom{q}{2}$ for every $1\le i<j\le q$, then the variation distance between  $\mu$ and $\nu$ is bounded by $\epsilon$.
\end{lemma}

\begin{definition}[$q$-statistic]\label{def:q-statistic} The q-statistic of a graph $G$ is the probability distribution over all (labeled) graphs with q vertices that result from picking at random q distinct vertices of G and considering the induced subgraph.
For a given graph $H$ we denote the probability for obtaining $H$ when drawing a graph according to the q-statistic by $\Pr_{G}(H)$.
\end{definition}

\begin{definition}\label{def:signature}
For an equipartition $A = \{V_{1},\dots, V_{t}\}$ of $G$, and a signature $S = (\eta_{i,j} )_{1\le i<j\le t}$ of $A$, the perceived q-statistic according to $S$ is the following distribution $\Pr_S$ over labelled graphs with $q$ vertices $v_1,\dots,v_q$.
Start by choosing a uniformly random sequence without repetitions of indices $i_{1},\dots, i_{q}$ from $1,\dots,t.$
Then, independently, take every $v_{k}v_{l}$ for $k<l$ to be an edge with probability $\eta_{i_{k},i_{l}}$if $i_{k}<i_{l}$ and with probability  $\eta_{i_{l},i_{k}}$if $i_{l}<i_{k}$. Then $\Pr_S(H)$ is defined as the probability that the resulting labelled graph equals $H$.
\end{definition}

The following lemma will replace Lemma \ref{lem:induced} in the proof of Theorem \ref{thm:main_gen}.

\begin{lemma}[see \cite{GT}]\label{lem:cannon_tets}
If there is an $\eps$-test for a graph property $\mathcal{P}$ that makes $Q = Q(\eps)$ edge queries, then there exists an appropriate family $\mathcal{H}$ of labeled graphs on $q=2Q$ vertices such that any graph $G$ which satisfies $\mathcal{P}$, satisfies also $\Pr_{G}(\mathcal{H})\ge\frac{2}{3}$, and any graph $G$ that is $\eps$-far from satisfying $\mathcal{P}$, satisfies also $\Pr_{G}(\mathcal{H})<\frac{1}{3}$.
\end{lemma}

We now introduce a variant of Lemma \ref{lem:sample_with_sig} that is suited for the proof of Theorem \ref{thm:main_gen}.

\begin{lemma}\label{lem:sample_with_sig_gen}
For every $q$, $\eps$ there are $\g=\g_{\ref{lem:sample_with_sig_gen}}(q,\eps),\ r=r_{\ref{lem:sample_with_sig_gen}}(q,\eps)$ so that
$$
\g=\mathrm{poly}(\eps\cdot2^{-q^2}),~~ r=\mathrm{poly}(1/\eps\cdot2^{q^2})
$$
and for every $\g$-signature $S$ of a $\g$-FK-regular equipartition $A$ into $t\ge r$ sets, of a graph $G$ on $n\ge N_{\ref{lem:sample_with_sig_gen}}(q,\eps,t)=\mathrm{poly}(t/\eps)2^{\mathrm{poly}(q)}$ vertices,
we have $|\Pr_{S}-\Pr_{G}|\le\eps$, where $\Pr_{G}$ is the $q$-statistic and $\Pr_{S}$ is the perceived $q$-statistic according to $S$.
\end{lemma}

\begin{proof}
We set
$$
\gamma=\frac{\eps}{56q^{2}2^{\binom{q}{2}}}=\mathrm{poly}(\eps\cdot2^{-q^2}),~~r\ge\frac{2}{\g}=\mathrm{poly}(1/\eps\cdot2^{q^2}),~~N = \frac{2t}{\g}\;.
$$
Let $\mathcal{H}$ be any family of $q$-vertex graphs.
We need to show that $|\Pr_{G}(\mathcal{H})-\Pr_{S}(\mathcal{H})| \leq \eps$.
Note that for every graph $H$ on $q$ vertices, it follows that $\Pr_{G}(H) = \frac{n^q\cdot(n-q)!}{n!}\cdot ind(H,G)$.
Let $S_A$ be the $0$-signature of $A$. It follows that $\Pr_{S_A}(H)=\frac{n^q\cdot(n-q)!}{n!}\cdot ind(H,S_{A})$, for every graph $H\in \mathcal{H}$.
By Lemma \ref{lem:blow_up_graph}, we have that
$$|ind(H,G_{A})-ind(H,S_{A})|\le \frac{2q^2}{t}+\frac{2tq}{n}\le 2q^2\g $$
for every $H\in\mathcal{H}$.
Also, by Lemma \ref{prop:small_var_under_regularity}, we have
$$|ind(H,G_{A})-ind(H,G)|\le 2q^2\g $$ for every $H\in\mathcal{H}$.
Thus, by triangle inequality, we conclude that $$|\text{Pr}_G(H)-\text{Pr}_{S_{A}}(H)|\le 4q^2\g \cdot \frac{n^q\cdot(n-q)!}{n!} \le 8q^2\g \le \frac{\varepsilon}{7\cdot 2^{\binom{q}{2}}}$$ for every $H\in\mathcal{H}$.
By counting over all $2^{\binom{q}{2}}$ options of $H$ on $q$ vertices, we get $|\Pr_{G}(\mathcal{H})-\Pr_{S_A}(\mathcal{H})|<\frac{\eps}{7}$.
Denote $S = (\eta_{i,j})_{1 \leq i < j \leq t}$.
We have $|\eta_{i,j}-d(V_{i},V_{j})|<\gamma\le\frac{\eps}{7q^{2}}$ for all pairs but $\gamma\binom{t}{2}$, because $S$ is a $\gamma$-signature of $A$.
Hence, with probability at least $1-\frac{4\eps}{7}$ we will sample distinct indices $i_{1},\dots i_{q}$ from $1,\dots,t$ such that every pair $i,j \in \{i_1,\dots,i_q\}$ satisfies $|\eta_{i,j}-d(V_{i},V_{j})|<\frac{\eps}{7q^{2}}$.
Thus by Lemmas \ref{lem:sub_var_dist} and \ref{lem:graph_var_dist}, we get $|\Pr_{S}(\mathcal{H})-\Pr_{S_A}(\mathcal{H})|<\frac{5\eps}{7}$.
All together, $|\Pr_{G}(\mathcal{H})-\Pr_{S}(\mathcal{H})|<\frac{6\eps}{7}<\eps$ as requested.
\end{proof}


We now introduce a variant of Lemma \ref{lem:info_from_sig} that is suited for the proof of Theorem \ref{thm:main_gen}.

\begin{lemma}\label{lem:info_from_sig_gen}
For every $q$ and $\eps$ there exist $\g=\g_{\ref{lem:info_from_sig_gen}}(q,\eps)$, $s=s_{\ref{lem:info_from_sig_gen}}(q,\eps)$ and $f_{\ref{lem:info_from_sig_gen}}^{(q,\eps)}:\mathbb{N}\to\mathbb{N}$, such that
$$
\g=\mathrm{poly}(\eps\cdot2^{-q^2}),~~
s=\mathrm{poly}\Big(\frac{2^{q^2}}{\eps}\Big),~~f_{\ref{lem:info_from_sig_gen}}^{(q,\eps)}(x)=x\cdot2^{\mathrm{poly}\big(\frac{2^{q^2}}{\eps}\big)}
$$
with the following property. For every family $\mathcal{H}$ of graphs with $q$ vertices, there exists a deterministic algorithm, that receives as an input a $\g$-signature $S$ of an $(f,\g)$-final partition $A$ into $t\ge s$ sets of a graph $G$ with $n\ge N_{\ref{lem:info_from_sig_gen}}(q,\eps,t)=t\cdot 2^{\mathrm{poly}(1/\eps)\cdot2^{\mathrm{poly}(q)}}$ vertices and distinguishes given any $\alpha$ between the following two cases:

\begin{enumerate}
	\item[(i)] $G$ is $(\alpha-\eps)$-close to some graph $G^{\prime}$ for which $\Pr_{G^{\prime}}(\mathcal{H})\ge\frac{2}{3}$.
	\item[(ii)] $G$ is $\alpha$-far from every $G^{\prime}$ for which $\Pr_{G^{\prime}}(\mathcal{H})\geq\frac{1}{3}$.
\end{enumerate}
\end{lemma}

\begin{proof}
We define the following parameters, as we did in the proof of Lemma \ref{lem:info_from_sig}.
$$
\gamma_0=\min\left\{\eps/10,\gamma_{\ref{lem:sample_with_sig_gen}}(q,\frac{1}{12}),\gamma_{\ref{lem:base_lemma}}(\min\{\eps/2,\gamma_{\ref{lem:sample_with_sig_gen}}(q,\frac{1}{12})\})\right\}=
\mathrm{poly}(\eps\cdot2^{-q^2}),
$$
and further define,
$$\g=\g_{\ref{lem:resp_reg}}(\g_{0})=\mathrm{poly}(\g_{0})= \mathrm{poly}(\eps\cdot2^{-q^2})$$
$$s=\max\{r_{\ref{lem:sample_with_sig_gen}}(q,1/12),r_{\ref{prop:dist_sig}}(\eps/10)\}=\mathrm{poly}\Big(\frac{2^{q^2}}{\eps}\Big)$$
$$T=T_{\ref{lem:resp_reg}}(t,\g_0)=t\cdot2^{\text{poly}(1/\eps)\cdot2^{\mathrm{poly}(q)}}$$
$$f=f_{\ref{lem:resp_reg}}^{(\g_{0})}(x)=x\cdot2^{\mathrm{poly}(1/\g_{0})}=x\cdot2^{\mathrm{poly}\big(\frac{2^{q^2}}{\eps}\big)}$$
$$
N=\max\{N_{\ref{lem:sample_with_sig_gen}}(q,1/12,t)),N_{\ref{prop:dist_sig}}(\eps/10,T_{\ref{lem:resp_reg}}(t,\g_{0})),N_{\ref{lem:resp_reg}}(s,\g_0),N_{\ref{lem:base_lemma}}(\g_0,t)\} = t\cdot 2^{\mathrm{poly}(1/\eps)\cdot2^{\mathrm{poly}(q)}}\;
$$
Given a family of graphs ${\cal H}$ on at most $q$ vertices, we defined a family of signatures as follows
$$
\mathcal{C}'_{{\cal H},T}=\{C:|C|\leq T ~\mbox{and}~\Pr_{C}(\mathcal{H})\ge 1/2\}\;.
$$
In order for $\mathcal{C}'_{{\cal H},T}$ to be finite, we only put in it signatures $C$ with edge weights $\eta_{i,j}$ that are integer multiples of $\beta=\frac{\eps}{12q^2}$.

Suppose now that we are given a $\gamma$-signature $S$ of some $(f,\gamma)$-final (with the above define $f,\gamma$) partition $A$ of a graph $G$, so that $S$ has $t\geq s$ parts and $G$ has at least $N$ vertices. Define $\mathcal{S}_{T}$ to be the set of all signatures on up to $T$ parts that are extensions of $S$.
The algorithm checks if there are $S'\in \mathcal{S}_{T}$ and $C\in {\cal C}'_{{\cal H},T}$ satisfying $d_1(S',C)\le \alpha-\frac{\eps}{2}$. If there is such a pair, the algorithm says that case $(i)$ holds, otherwise it says that case $(ii)$ holds. We now prove the correctness of the algorithm.

\paragraph{Proof of first direction:} Suppose that $G'$ is some graph $(\alpha-\eps)$-close to $G$, and for which $\Pr_{G'}(\mathcal{H})\ge2/3$. We will show that the algorithm will declare that case $(i)$ holds. Recall that $A$ is an $(f,\g)$-final partition of $G$ into $t\ge s$ sets, and that $S$ is a $\g$-signature of $A$. By Lemma \ref{lem:resp_reg}, there exists a refinement $A'$ of $A$ into at most $T$ sets so that $A'$ is $\g_{0}$-FK-regular$^{\star}$ for both $G$ {\bf and} $G^{\prime}$. Moreover, denoting by $S'$ the corresponding extension of $S$ to $A'$, we have that $S'$ is a $\g_{0}$-signature of $A'$ with respect to $G$. Note that $S'\in\mathcal{S}_{T}$.
By the choice of $\g_{0}$ this implies that $A'$ is $\g_{\ref{lem:sample_with_sig_gen}}(q,1/12)$-FK-regular$^{\star}$ for both $G$ and $G^{\prime}$, and that $S'$ is a $\frac{1}{10}\eps$-signature of $A'$ with respect to $G$.
Let $C'$ be the $0$-signature of $A'$ over $G^{\prime}$. Lemma \ref{lem:sample_with_sig_gen} (using $A'$ and $G^{\prime}$) implies that $|\Pr_{G'}(\mathcal{H})-\Pr_{C'}(\mathcal{H})|<\frac{1}{12}$. Thus $\Pr_{C'}(\mathcal{H})\ge2/3-1/12\ge7/12$. Clearly, there is a signature $C$ so that all $C$'s weights are constant multiples of $\beta$, and $|c_{ij}-c'_{ij}|\le \beta$. Thus, by Lemma \ref{lem:graph_var_dist}, we infer that $\Pr_{C}(\mathcal{H})\ge7/12-1/12\ge1/2$
 so, $C\in{\cal C}'_{{\cal H},T}$. In addition, by Claim \ref{prop:dist_sig} (since $A'$ has at least $r_{\ref{prop:dist_sig}}(\frac{\eps}{10})$ sets and assuming that $n$ is large enough), $d_1(S',C)\le(\alpha-\frac{\eps}{2})$  on account of $G$ and $G^{\prime}$ being $(\alpha-\eps)$-close graphs. Thus, $S'$ and $C$ provide a witness that the procedure above accepts $G$.

\paragraph{Proof of second direction:}
Suppose the algorithm declares that case $(i)$ holds. We show that in this case, there is a graph $G'$, which is $\alpha$-close to $G$, and satisfies, $\Pr_{G'}(\mathcal{H})>1/3$.

Indeed, if the algorithm declared that case $(i)$ holds then there are signatures $S'\in\mathcal{S}_{T}$ and $C\in{\cal C}'_{{\cal H},T}$ satisfying $d_1(S',C)\le \alpha-\frac{\eps}{2}$. As $S'\in\mathcal{S}_{T}$, there is a refinement $A'$ of $A$, so that $S'$ is the extension of $S$ according to $A'$. Lemma \ref{lem:resp_reg} (regarding $A'$ as a possible refinement of $A$ with respect to $G$) asserts that $S'$ is a $\g_{0}$-signature of $A'$ (with respect to $G$), which by the choice of $\g_{0}$ means that it is a $\g_{\ref{lem:base_lemma}}(\min\{\frac{\eps}{2},\g_{\ref{lem:sample_with_sig_gen}}(q,\frac{1}{12})\})$-signature for $A'$ with respect to $G$.
Now, Lemma \ref{lem:base_lemma} (applied with $A'$ as the $\g_{0}$-FK-regular$^{\star}$ partition of $G$, and with $S'$ as $S$ and $C$ as $S'$) implies that there is a graph $G^{\prime}$ that is $(\alpha-\frac{\eps}{2}+\frac{\eps}{2})$-close to $G$, namely $\alpha$-close to $G$, and for which $C$ is a $\g_{\ref{lem:sample_with_sig_gen}}(q,\frac{1}{12})$-signature of $A'$, which in turn is $\g_{\ref{lem:sample_with_sig_gen}}(q,\frac{1}{12})$-FK-regular over $G^{\prime}$ . Lemma  \ref{lem:sample_with_sig_gen} implies that $|\Pr_{G'}(\mathcal{H})-\Pr_{C}(\mathcal{H})|<\frac{1}{12}$. Thus $\Pr_{G'}(\mathcal{H})\ge1/2-1/12>1/3$ as required. Hence we have found the required $G'$.
\end{proof}

We are ready to derive Theorem \ref{thm:main_gen} from Lemmas \ref{lem:final_part_sig} and \ref{lem:info_from_sig_gen}. The proof is similar
to the way we derived Theorem \ref{thm:main} from Lemmas \ref{lem:final_part_sig} and \ref{lem:info_from_sig}.

\begin{proof}[Proof (of Theorem \ref{thm:main_gen}):]
Suppose $\mathcal{P}$ is a testable graph property, and let $\alpha, \eps>0$ the constants for which we would like to $(\alpha,\eps)$-estimate $\mathcal{P}$. As $\mathcal{P}$ is $\eps/2$-testable, there is a testing algorithm that given $\eps/2$ and a graph $G$ makes $Q(\eps/2)$ queries. Lemma \ref{lem:cannon_tets} asserts us that there exists a constant $q=q_{\ref{lem:cannon_tets}}(\eps/2)=\text{poly}(Q(\eps/2))$ and a family of graphs $\mathcal{H}$ on $q$ vertices such that for every $G\in\mathcal{P},\Pr_G(\mathcal{H})\ge 2/3$, and for every graph $G$ that is $\eps/2$-far from satisfying $\mathcal{P}$, $\Pr_G(\mathcal{H})\le1/3$.
We thus set,
$$
\g=\gamma_{\ref{lem:info_from_sig_gen}}(q,\eps/2)=\mathrm{poly}(\eps\cdot2^{-q^2}),
$$
$$
f=f_{\ref{lem:info_from_sig_gen}}^{(q,\eps/2)}(x)=x\cdot2^{\mathrm{poly}\big(\frac{2^{q^2}}{\eps}\big)},
$$
$$
k=s_{\ref{lem:info_from_sig_gen}}(q,\eps/2)=\mathrm{poly}\Big(\frac{2^{q^2}}{\eps}\Big)
$$
Now, define $\zeta =  \eps\cdot 2^{-q^2}$ and apply the algorithm provided by Lemma \ref{lem:final_part_sig} with parameters $k,\zeta,\g,f$ on the input graph $G$. This algorithm makes up to $q_{\ref{lem:final_part_sig}}(k,\zeta)$ queries, and from the assumptions of the lemma, we have that
$$
q_{\ref{lem:final_part_sig}}(k,\zeta)=\text{poly}(k)2^{\text{poly}(1/\zeta)}=2^{\text{poly}(1/\eps)\cdot2^{\text{poly}(q)}}\;.
$$
With probability at least $2/3$ the algorithm returns a $\g$-signature of an $(f,\g)$-final equipartition of $G$ with at least $s$ and at most $T_{\ref{lem:final_part_sig}}(k,\g,f)$ sets.
We now apply the algorithm provided by Lemma \ref{lem:info_from_sig_gen} with parameters $q,\eps/2,\alpha-\eps/2$, to the signature $S$. Due to the choice of parameters, it is guaranteed by Lemma  \ref{lem:info_from_sig_gen} that we can distinguish between the case that there is a graph $G'$ that is $(\alpha-\eps)$-close to $G$ and for which $\Pr_{G'}(\mathcal{H})\ge 2/3$, and that $G$ is $(\alpha-\eps/2)$- far from any $G'$ for which $\Pr_{G'}(\mathcal{H})> 1/3$. In the first case $G$ is accepted, and in the second case it is rejected.
For the above to work we require $N\ge\max\{N_{\ref{lem:final_part_sig}}(k,\g,f),N_{\ref{lem:info_from_sig_gen}}(q,\frac{\eps}{2},T_{\ref{lem:final_part_sig}}(k,\g,f))\}$.
For a smaller $n$ we can just read the entire input and compute its distance from the property to be
estimated, with query complexity $N^2\le 2^{\text{poly}(1/\eps)\cdot2^{\text{poly}(q)}}$.
We now claim that the algorithm above is indeed an $(\alpha,\eps)$-estimation algorithm for $\mathcal{P}$ for every $n\ge N$.

If $G$ is $(\alpha-\eps)$-close to $\mathcal{P}$, then it is also $(\alpha-\eps)$-close to some graph $G'$ for which $\Pr_{G'}(\mathcal{H})\ge 2/3$. And so the first case above will hold as long as $S$ is in fact a $\g$-signature of an $(f, \g)$-final partition of $G$, which happens with probability at least $2/3$. Thus $G$ is accepted with probability at least $2/3$.

On the other hand, if $G$ is $\alpha$-far from $\mathcal{P}$, then by the triangle inequality it is $(\alpha-\eps/2)$-far from any $G'$ for which $\Pr_{G'}(\mathcal{H})> 1/3$. And so the second case above will hold as long as $S$ is in fact a $\g$-signature of an $(f, \g)$-final partition of $G$, which happens with probability at least $2/3$. Thus $G$ is rejected with probability at least $2/3$.
\end{proof}

\begin{proof}[Proof (of Proposition \ref{easygeneral}):]
Recall that we use $\text{dist}_\mathcal{P}(G)$ to denote the minimal number of edge additions/deletions one needs to perform in order to turn $G$ into a graph satisfying $\mathcal{P}$, normalized by $|V(G)|^2$.
If $\mathcal{P}$ is natural, then there is a positive sequence $\{\eps_t\}_{t=1}^{\infty}$ tending to zero,
so that for every $\eps_t$ there is a sequence of integers $\{n_k\}_{k=1}^{\infty}$
so that for every $k$ there is a graph $G_k$ on $n_k$ vertices satisfying $\eps_t \le \text{dist}_{\mathcal{P}}(G_k) \le 1$.
We will now show that for every $t$ and every $k$, every $\eps_t$-tester for ${\cal P}$
must make at least $\frac{1}{40\eps_t}$ edge queries when operating on $n_k$-vertex graphs.

Fix a $t$ as above. The key observation is that the graphs $G_k$ defined above can be further assumed to satisfy $\text{dist}_\mathcal{P}(G_k)\le 2\eps_t$.
Indeed, since for large enough $n_k$, we have that $\eps_t\cdot n_k^2\ge 1$, modifying one edge of $G_k$, changes $\text{dist}_\mathcal{P}(G)$ (additively) by at most $\eps_t$.
Hence, we can start with the graph $G_k$ which satisfies $\eps_t \le \text{dist}_\mathcal{P}(G) \le 1$, and then modify its edges one by one until we obtain a new graph $G_k^\prime$ satisfying $\eps_t \le \text{dist}_\mathcal{P}(G_k^{\prime})\le 2\eps_t$.
For each such $G_k$ on $n_k$ vertices, since $\text{dist}_\mathcal{P}(G_k)\le 2\eps_t$, there is a graph $H_k$ on $n_k$ vertices which satisfies ${\cal P}$ and such that $G_k\triangle H_k$ (the symmetric difference of the graphs) has at most $2\eps_t \cdot n^2_k$ edges. The key observation now is that an $\eps_t$-tester for ${\cal P}$ which makes less than
$\frac{1}{40\eps_t}$ edge queries has probability at most $1/10$ of querying one of the edges of $G_k\triangle H_k$.
Hence, the probability that it distinguishes between $G_k$ and $H_k$ is at most $1/10$, and so it cannot be an $\eps_t$-tester for ${\cal P}$.
\end{proof}

\bibliographystyle{plain}

\end{document}